\documentclass[reqno]{amsart}

\usepackage{amsmath,amssymb}
\usepackage{graphicx,color}

\usepackage{amsthm}

\newtheorem{theorem}{Theorem}[section]

\newtheorem{lemma}[theorem]{Lemma}

\newtheorem{corollary}[theorem]{Corollary}
\theoremstyle{definition}

\newtheorem{example}[theorem]{Example}
\newtheorem{remark}[theorem]{Remark}

\newtheorem{remarks}[theorem]{Remarks}

\newcommand{\norm}[2][\relax]{
	\ifx#1\relax \ensuremath{\left\Vert#2\right\Vert}
	\else \ensuremath{\left\Vert#2\right\Vert_{#1}}
	\fi}
\newcommand{\Norm}[2][\relax]{
	\ifx#1\relax \ensuremath{\left\Vert#2\right\Vert}
	\else \ensuremath{\left\Vert#2\right\Vert_{#1}}
	\fi}

\def\R{\mathbb{R}}

\def\Z{\mathbb{Z}}

\def\BFS{{\mathrm{E}}}
\def\BBFS{{\mathbf{E}}}

\def\cL{{\mathcal L}}
\def\cM{{\mathcal M}}
\def\cN{{\mathcal N}}

\newcommand{\ud}{\mathrm{d}}

\newcommand{\X}{{X}}
\newcommand{\Y}{{Y}}

\newcommand{\st}{\,|\,}

\DeclareMathOperator{\Dom}{dom}  
\DeclareMathOperator{\Rg}{ran}
\DeclareMathOperator{\Var}{Var}
\newcommand{\dom}[1]{\Dom {#1}}
\newcommand{\rg}[1]{\Rg {#1}}

\def\tr\XYetaq{Tr(\X,Y, \eta ,q)}
\def\tr\XY1-tq{Tr(\X,Y, 1-\theta ,q)}
\def\tr\XYtq{Tr(\X,Y, \theta,q)}

\definecolor{darkred}{rgb}{0.7,0.1,0.1}

\newcounter{aufzi}
\renewcommand\theaufzi{(\alph{aufzi})}
\newenvironment{aufzi}{\begin{list}{ {\upshape\theaufzi }}{
         \usecounter{aufzi}
         \topsep1ex
         \parsep0cm
         \itemsep1ex
         \leftmargin0.8cm
         \labelwidth0.5cm
         \labelsep0.3cm
}}
{\end{list}}

\newcounter{aufzii}
\renewcommand\theaufzii{(\roman{aufzii})}

\newcounter{aufziii}
\renewcommand\theaufziii{(\arabic{aufziii})}

\newcounter{aufziv}
\renewcommand\theaufziv{\Roman{aufziv}}

\title[Real interpolation of functions]{Real interpolation of functions with applications to accretive operators on Banach spaces}

\author{Ralph Chill}

\address{R.~Chill, Institut f\"ur Analysis, Fakult\"at Mathematik, TU Dresden, 01062 Dresden, Germany}
\email{ralph.chill@tu-dresden.de}

\author{Praveen Sharma}

\address{P.~Sharma, Department of Mathematics,
University of Delhi South Campus,
Benito Juarez Road,
New Delhi 110021, India}
\email{sharmapraveen7089@gmail.com}

\author{Sachi Srivastava}

\address{S.~Srivastava, Department of Mathematics,
University of Delhi South Campus,
Benito Juarez Road,
New Delhi 110021, India}
\email{sachi\_srivastava@yahoo.com}

\thanks{This project is supported by the VAJRA scheme VJR/2018/000127, of DST-SERB, Govt. of India}

\begin{document}

	\date{\today}
	
	
	
	\begin{abstract} 
		We study real interpolation, but instead of interpolating between Banach spaces, we interpolate between general functions taking values in $[0,\infty].$  We show the equivalence of the mean method and the $K$-method and apply the general theory to interpolation between the norm on a Banach space and the set norm associated with an m-accretive operator on such a space. 
	\end{abstract}

	\renewcommand{\subjclassname}{\textup{2010} Mathematics Subject Classification}
	
	\maketitle
	
\section{Introduction}

The interpolation theory of Banach spaces has been developed in the early 1960s. Real interpolation spaces play a prominent role because they appear naturally as trace spaces of, for instance, classical Sobolev spaces on domains and are thus connected to the regularity theory of elliptic or parabolic boundary value problems. In the abstract setting, they appear as trace spaces of certain weighted, vector-valued Sobolev spaces on the real half-line and find here their connection to abstract Cauchy problems in Banach spaces, namely as the spaces of initial values for which the solutions have some described time regularity; see, for example, the monographs by Bergh \& L\"ofstr\"om \cite{BeLo76}, Lunardi \cite{Lu95,Lu09} and Triebel \cite{Tr95}.

Given an interpolation couple $(\X_0 ,\X_1 )$ of Banach spaces, $p\in (1,\infty )$, $\theta\in (0,1)$, one defines the real interpolation space $(\X_0 ,\X_1 )_{\theta ,p}$ to be the space of all elements $x\in \X_0 +\X_1$ such that 
\[
 t\mapsto \frac{K(x,t)}{t} \in L^p(0,\infty , t^{p(1-\theta) - 1}\ud t) ,
\]
where $K(x,t) = K(x,t,\X_0 , \X_1 )$ is the $K$-functional given by
\begin{equation} \label{lineark}
K(x,t) := \inf_{ x = x_0 + x_1 \atop x_0 \in \X_0, x_1 \in \X_1 } ( \norm[\X_0]{x_0} + t \norm[X_1]{x_1} ).
\end{equation}
The space $(\X_0 ,\X_1 )_{\theta ,p}$ is a Banach space for the norm 
\[
 \norm[(\X_0 ,\X_1)_{\theta ,p}]{x} := \Norm[L^p(0,\infty , t^{p(1-\theta) - 1}\ud t)]{t\mapsto \frac{K(x,t)}{t} } ,
\]
and it is indeed an interpolation space in the sense that every linear operator $T:\X_0 \cap \X_1 \to \X_0 + \X_1$ which maps $\X_0$ and $\X_1$ boundedly into itself also maps the space $(\X_0 ,\X_1 )_{\theta ,p}$ boundedly into itself. 

In this article, we change the perspective on interpolation in the following way. Instead of saying that we interpolate between Banach spaces, we rather say that we interpolate between norms on $\X = \X_0 + \X_1$, with the small difference that our norms are allowed to take values in $[0,\infty ]$, that is, the value $\infty$ is possible.  The resulting interpolation function is again a norm and its effective domain (the set of elements in which the norm is finite) is the interpolation space. 

Having changed the perspective, we extend the theory by not only interpolating between norms, but between general functions with values in $[0,\infty ]$. This could for example be powers of norms (compare with Peetre \cite{Pe70}, Tartar \cite{Tar72} and Bergh \& L\"ofstr\"om \cite[Section 3.11]{BeLo76}), but one can think of much more general functions. 

In this article, we review some classical interpolation methods in the context of interpolation couples of functions. We show that effective domains of interpolation functions (defined by a mean method) are left invariant by certain nonlinear operators. This could be a justification to call the effective domains interpolation sets, but an example shows that the effective domains of interpolation functions rather depend on the interpolation couple of functions than on their domains. Further, we  show that our mean method and the $K$-method are equivalent. For the definition of interpolation functions, instead of just  relying on the polynomially weighted $L^p$-spaces, we use a larger class of Banach function spaces between $L^1 (0,\infty )$ and $L^\infty (0,\infty )$. Bennett, for example, extended the $K$-method by replacing the classical polynomially weighted $L^p$ spaces by {\em unweighted} rearrangement invariant Banach function spaces \cite{Be74II}. On the other hand, several authors studied extensions of the $K$-method using more general weights than power weights, staying however within the $L^p$ scale; see for example Kalugina \cite{Ka75} (quasipower weights), Sagher \cite{Sa81} (weights of Calder\'on type), and more recently Bastero, Milman \& Ruiz \cite{BaMiRu01} (Calder\'on weights). Real interpolation spaces on the basis of general Banach function spaces have been studied in Bennett \& Sharpley \cite{BeSh88} and Brudny{\u\i} \& Krugljak \cite{BrKr91}; see also Chill \& Kr\'ol \cite{ChKr17} in this context.

We then apply the general theory to interpolation between the norm on a Banach space and the set norm associated with a nonlinear $m$-accretive operator on that space. Such operators are negative generators of nonlinear, strongly continuous semigroups. Interpolating between the norm on a Banach space and the set norm of an $m$-accretive operator can be seen as interpolating between the domain of the operator and the closure of the domain. Such a theory has been developed in 1973 and 1981 in two theses by D. Brezis \cite{Bre74a} (who considered $m$-accretive operators on Hilbert spaces) and A. Dufetel \cite{Df81a} (who generalized this by considering $m$-accretive operators on Banach spaces). Both these very interesting works lead to the articles \cite{Bre76} and \cite{Df81}, but it seems that perhaps they did not get as much attention as they deserved.  We show that the interpolation functions defined by the mean method and by the $K$-method are equivalent to other functions defined via the resolvent and the semigroup, and we thus characterize their effective domains. We show that semigroup orbits corresponding to initial values in the effective domain of interpolation functions have some prescribed  regularity between Lipschitz continuity and mere continuity. As a consequence, such orbits are H\"older type continuous. We then prove an invariance of interpolation sets under perturbations, and we give an additional description of interpolation sets when the underlying semigroup is regularizing ($=$ analytic in the case of linear $C_0$-semigroups). Finally, in the special case when the $m$-accretive operator is the subgradient of a proper, lower semicontinuous and convex function on a Hilbert space, we characterize the effective domain of associated interpolation functions in terms of interpolation functions between the norm of the Hilbert space and the square root of the energy function, and we apply this to prove a regularity result for a parabolic equation involving the $q$-Laplace operator.

\section{Banach function spaces} \label{sec.definition}

\subsection{Banach function spaces} The interpolation theory is based on the use of function spaces over the interval $(0,\infty )$. We denote by $L^0 (0,\infty )$ the set of all (equivalence classes of) complex, measurable functions on $(0,\infty )$  and we let $L^1 (0,\infty )$ be the space of all Lebesgue integrable functions on $(0,\infty )$. Accordingly,  $W^{1,1} (0,\infty )$ is the Sobolev space of all integrable functions on $(0,\infty )$ which admit a distributional derivative in $L^1 (0,\infty )$. 

Following Brudny{\u\i} \& Krugljak \cite[Definition 2.6.3, p. 246]{BrKr91}, a Banach space $\BFS$ is called a {\it Banach function space over $(0,\infty )$} if
\begin{enumerate}
\item $L^1\cap L^{\infty} (0,\infty ) \subseteq \BFS \subseteq L^1 +L^\infty (0,\infty )$, and 
\item $\BFS$ is an order ideal of $L^1 +L^\infty (0,\infty )$, that is, for every $f$, $g\in L^1 +L^\infty (0,\infty )$
\[
|g|\leq |f| \text{ and } f\in \BFS \quad \Rightarrow \quad g \in \BFS \text{ and } \Norm[\BFS]{g} \leq \Norm[\BFS]{f} .
\]
\end{enumerate}
 
This definition is consistent with Meyer-Nieberg \cite[Definition 1.1.5, p. 6]{MN91}, for instance, where it is used in the sense of complete K\"othe function space, but it differs from the definition in Bennett \& Sharpley \cite[Definition 1.3, p. 3]{BeSh88}, where in addition a Fatou property is required to hold. Note that in Brudny{\u\i} \& Krugljak \cite{BrKr91}, Banach function spaces are called Banach lattices.  

Typical Banach function spaces are the polynomially weighted spaces 
\[
\BFS = L^p(0,\infty , t^{p(1-\theta) - 1}\ud t)  \mbox{ for } p\in [1,\infty ) \mbox{ and } \theta\in (0,1) .
\]
The polynomial weights 
\[
 w_{\theta ,p} (t) = t^{p(1-\theta) - 1} \quad (\theta\in (0,1)) 
\]
appearing here are (restrictions of) typical Muckenhoupt $A_p$-weights on the real line for each $p\in [1,\infty )$. For an arbitrary Muckenhoupt $A_p$-weight $w$ on the real line,  the space
\[
 \BFS = L^p (0,\infty , w(t) \ud t)
\]
is a Banach function space, too. Other examples of Banach function spaces are the Lorentz spaces $L^{p,q} (0,\infty )$, the Orlicz spaces $L^\Phi (0,\infty )$, general rearrangement invariant Banach function spaces over $(0,\infty )$, weighted versions of these spaces as well as intersections and sums of such spaces. 

Given a Banach space $\X$ and a Banach function space $\BFS$ over $(0,\infty )$, we define the vector-valued variant of $\BFS$ by setting
\[
\BFS (\X ) :=\left\{ f: (0,\infty ) \rightarrow \X \text{ measurable} \st  \norm[\X]{f} \in \BFS \right\} .
\]
This space is a Banach space for the norm $\Norm[\BFS(\X)]{f} :=\Norm[\BFS]{\norm[\X]{f}}$ (for simplicity, as it is usually done, we identify the space of measurable functions with the space of equivalence classes of measurable functions).

\subsection{The Hardy operator.} The {\em Hardy operator} is the integral operator which assigns to every function $f\in L^1 +L^\infty (0,\infty )$ the function $Pf$ given by
	\begin{equation}\label{hardyop}
	Pf(t):=\frac{1}{t}\int_0^t f (s)\,\ud s  \quad (f\in L^1 +L^\infty (0,\infty ), \, t\in (0,\infty )) . 
\end{equation}
In a few of our main results, we need that the Hardy operator maps a Banach function space $\BFS$ into itself; by the closed graph theorem, the Hardy operator is then bounded on $\BFS$. The Hardy operator maps $L^1 (0,\infty )$ into the larger Lorentz space $L^{1,\infty} (0,\infty )$ (the weak-$L^1$ space), which in itself is not an intermediate space between $L^1$ and $L^\infty$, and which is only a quasi Banach space. Moreover, the Hardy operator obviously maps $L^\infty (0,\infty )$ into itself. By the Marcinkiewicz interpolation theorem \cite[Theorem 4.13, p.225]{BeSh88}, for every $p\in (1,\infty ]$ and every $q\in [1,\infty ]$ the Hardy operator therefore maps the Lorentz space $L^{p,q} (0,\infty )$ into itself, and in particular it maps $L^p (0,\infty )$ into itself. More generally, by \cite[Theorem 5.15, p.150]{BeSh88}, the Hardy operator is bounded on an unweighted rearrangement invariant Banach function space over $(0,\infty )$ if and only if the lower Boyd index of this Banach function space is strictly larger than $1$ (here, in this reformulation of \cite{BeSh88}, the Boyd indices are elements of $[1, \infty ]$ while in the definition of \cite{BeSh88} the Boyd indices are their reciprocals and thus elements of $[0, 1]$). Finally, by \cite[Theorem 7.3]{Du01}, the Hardy operator on a weighted $L^p$ space over $(0,\infty )$ ($p> 1$) is bounded if the weight is a Muckenhoupt $A_p$-weight. 

\subsection{Functions of bounded variation and absolutely continuous functions} 

Let $\X$ be a Banach space, and let $\tau\in (0,\infty ]$. For a function $f:[0,\tau )\to \X$ we define the {\em variation function} $\Var_f : [0,\tau ) \to [0,\infty] $ by
\[
Var_f (t) := \sup_{0\leq s_0 \leq \dots \leq s_n\leq t} \sum_{i=0}^{n-1} \norm[\X]{f(s_{i+1}) - f(s_i)} .
\] 
The variation function is increasing. The function $f$ is {\em of bounded variation} if its variation function $\Var_f$ is bounded on $[0,\tau )$.

A function $f:[0,\tau )\to \X$ is {\em absolutely continuous} if for every $\varepsilon >0$ there exists $\delta >0$ such that for every finite family $([a_i,b_i])_{1\leq i\leq n}$ of mutually disjoint subintervals of $[0,\tau )$ 
\[
\sum_{i=1}^n (b_i-a_i) \leq \delta \quad \Rightarrow \quad \sum_{i=1}^n (\Var_f (b_i) - \Var_f (a_i)) \leq \varepsilon .
\]
We denote by $AC([0,\tau );\X )$ the space of all absolutely continuous functions, and by $AC_{loc} ([0,\tau );\X )$ the space of all locally absolutely continuous functions, that is, of all functions $[0,\tau )\to\X$ which are absolutely continuous on every compact subinterval of $[0,\tau )$.

Every absolutely continuous function on a finite interval is of bounded variation. Moreover, if $f$ is absolutely continuous, then the variation function $\Var_f$ is absolutely continuous, too. By the Radon-Nikodym theorem, for every finite $\tau$, $AC ([0,\tau )) = W^{1,1} (0,\tau )$, that is, every absolutely continuous,  scalar valued function is weakly differentiable (in the sense of Sobolev) with weak derivative in $L^1 (0,\tau )$. For vector-valued functions, this equality is no longer true in general; the space $W^{1,1} (0,\tau ;\X)$ always is a closed subspace of $AC ([0,\tau );\X)$ (still, $\tau$ is finite), and for every $f\in W^{1,1} (0,\tau ;\X)$ one has 
	\begin{equation} \label{eq.dotf.varf}
	\int_0^t\norm[\X]{\dot f(s)}ds = \Var_f(t) ,
	\end{equation}
	but the inclusion of $W^{1,1} (0,\tau ;\X) $ into $AC ([0,\tau );\X)$ may be strict. A Banach space $\X$ has the {\em Radon-Nikodym property} if $W^{1,1} (0,\tau ;\X) = AC([0,\tau );\X)$, that is, every absolutely continuous function $f:[0,\tau )\to \X$ admits a weak derivative in $L^1 (0,\tau ;\X)$ (the property is independent of $\tau$). Every reflexive Banach space and every separable dual space has the Radon-Nikodym property (see \cite{ABHN01,DiUh77}). 
	
	Let $\BFS\subseteq L^1+L^\infty (0,\infty )$ be a Banach function space. We define the homogeneous Sobolev space
	\[
	\mathring{W}^{1,\BFS} (0,\infty ;\X) := \{ f\in W^{1,1}_{loc} (0,\infty ;\X) \st \norm[\X]{\dot f(\cdot )} \in \BFS \} 
	\]
	and the homogeneous space of absolutely continuous functions
	\[
	\mathring{AC}^\BFS ([0,\infty );\X) := \{ f\in AC_{loc} ([0,\infty );\X) \st \frac{\ud}{\ud t} \Var_f\in \BFS \} .
	\]
	By \eqref{eq.dotf.varf}, $\mathring{W}^{1,\BFS} (0,\infty ; \X)$ is a closed subspace of $\mathring{AC}^\BFS ([0,\infty );\X)$, and both spaces are equal if and only if $\X$ has the Radon-Nikodym property.  

\begin{lemma} \label{lem.ac}
Let $\BFS$ be a Banach function space over $(0,\infty )$ such that the Hardy operator $P$ is bounded on $\BFS$ and let $u \in \mathring{AC}^\BFS ([0,\infty );\X)$. Then
\[
t \mapsto \left\| \frac{u(t) - u(0)}{t} \right\|_\X \in \BFS .
\]
\end{lemma}

\begin{proof}
Since $u \in \mathring{AC}^\BFS ([0,\infty );\X)$ then $\Var_u \in \mathring{W}^{1,\BFS}(0,\infty )$. Hence, there is a $g \in \BFS$ such that
\[
\Var_u (t) = \int_0^t g(s) \, \ud s \quad \text{for } t \in [0,\infty ).
\]
For $t \in (0,\infty )$ we have
\begin{align*}
\frac{1}{t} \left\| u(t) - u(0) \right\|_\X &\le \frac{1}{t} \Var_u (t) = \frac{1}{t} \int_0^t g(s) \,\ud s = P g(t).
\end{align*}
Since the Hardy operator is bounded on $\BFS$, $Pg \in \BFS$ and the claim follows from the ideal property of $\BFS$.
\end{proof}

\section{Interpolation of functions}

Our basic setting for interpolation is the following. Let $\X$ be a Banach space. Throughout we consider measurable functions $\cN : \X\to [0,\infty ]$. Given such a function, we call 
\[
\dom{\cN} := \{ x\in \X \st \cN (x) <\infty \}
\]
its {\em effective domain}.

We call a pair $\cN = (\cN_0 ,\cN_1 )$ of measurable functions $\cN_0$, $\cN_1 : \X\to [0,\infty ]$ an {\em interpolation couple of functions}. 

\begin{example}[Linear interpolation] \label{ex.linear}
Let $(\X_0,\X_1)$ be an interpolation couple of Banach spaces (see, for example, \cite[Section 2.3]{BeLo76}), that is, both $\X_0$ and $\X_1$ are continuously embedded into a Hausdorff topological vector space. In this Hausdorff topological vector space one can define the {\em sum space} $\X := \X_0 +\X_1$, which is a Banach space for the norm
\[
\norm[\X]{x} = \inf \{ \norm[\X_0]{x_0} + \norm[\X_1]{x_1} \st x_i\in \X_i \text{ and } x = x_0 +x_1\} . 
\]
On the sum space we define two functions $\cN_0$, $\cN_1 : \X\to [0,\infty ]$ by
\[
\cN_i (x)  = \begin{cases} 
\norm[\X_i]{x} & \text{if } x\in \X_i \subseteq \X , \\[2mm]
\infty & \text{if } x\in \X\setminus \X_i . 
\end{cases}
\]
Both functions $\cN_0$ and $\cN_1$ satisfy all properties of a norm with the exception that both functions actually take values in $[0,\infty ]$. By abuse of language, we still call such functions {\em norms}.
The couple $\cN = (\cN_0 ,\cN_1)$ is an interpolation couple of functions. We shall see below that this interpolation couple of {\em norms} corresponds to the standard interpolation couple of {\em Banach spaces} from linear interpolation theory. 
\end{example}

\begin{example}[Nonlinear interpolation] \label{ex.nonlinear}
Still in the case, when $(\X_0,\X_1)$ is an interpolation couple of Banach spaces, one may also consider the functions $\cN_0$, $\cN_1 : \X\to [0,\infty ]$ given by
\[
\cN_i (x)  = \begin{cases} 
\norm[\X_i]{x}^{\alpha_i} & \text{if } x\in \X_i \subseteq \X , \\[2mm]
\infty & \text{if } x\in \X\setminus \X_i ,
\end{cases}
\]
where $\alpha_0$, $\alpha_1\in (0,\infty )$, that is, the $\cN_i$ are powers of norms. This situation was called {\em nonlinear interpolation} by Peetre \cite{Pe70} and Tartar \cite{Tar72,Tar72a}. The effective domains of these functions $\cN_i$ are still Banach spaces, but the functions themselves are only quasi-norms \cite[Section 3.11]{BeLo76}.
\end{example}

\begin{example}[Interpolation and $m$-accretive operators] \label{ex.accretive}
Let $\X$ be a Banach space and let $A\subseteq \X\times \X$ be an {\em $m$-accretive operator of type $\omega\in\R$}, that is, $A$ is {\em accretive of type $\omega$} in the sense that, for every $(x,f)$, $(\hat{x},\hat{f})\in A$ and for every $\lambda >0$ with $\lambda\omega <1$, 
\[
\norm[\X]{x-\hat{x} + \lambda (f-\hat{f})} \geq (1-\lambda\omega ) \, \norm[\X]{x-\hat{x}} ,  
\]
and, moreover, $A$ satisfies the   condition
\[
\rg {(I+\lambda A)} = \X \text{ for every } \lambda >0 \text{ with } \lambda\omega <1 .
\]
In Section 4 below, we consider the interpolation couple ${\cN}^{A} = (\cN_0 , \cN_1 )$, where $\cN_0 = \norm[\X]{\cdot}$ is the norm in $\X$ and 
\[
\cN_1 (x) = |A x| := \inf \{ \norm[\X]{f} \st (x,f)\in A\} ,
\]
is the set norm of $ A x$, with the usual interpretation $\inf \emptyset = \infty$. Note here, that $ Ax := \{ f\in\X \st (x,f)\in A\}$ is a subset of $\X$, and that for any subset $B\subseteq\X$, we call 
\[
|B| := \inf \{ \norm[\X]{x} \st x\in B\}
\]
its {\em set norm}.
\end{example}

\begin{example}[Interpolation and convex functions]
    Let ${\mathcal E} : H\to [0 , \infty ]$ be a proper, lower semicontinuous, convex function on a Hilbert space $H$. Proper means here that the effective domain of $\mathcal E$ is nonempty. Assume, for simplicity, that ${\mathcal E}$ admits a global minimizer $\bar{x}$ and that ${\mathcal E} (\bar{x}) =0$. Otherwise, replace the function ${\mathcal E}$ by the function ${\mathcal E} + \frac12 \norm[H]{\cdot}$, which, by lower semicontinuity and convexity, is {\em coercive} in the sense that its sublevel sets are bounded in $H$. By a well known result from the calculus of variations, this rescaled function thus admits a global minimizer $\bar{x}\in H$, and changing the rescaled function by an appropriate constant shows that its minimum is $0$. We consider now the interpolation couple $\cN = (\cN_0 ,\cN_1 )$, where $\cN_0 = \norm[H]{\cdot}$ is the norm in $H$ and the function $\cN_1$ is is the square root of${\mathcal E}$, that is, $\cN_1 = \sqrt{{\mathcal E} (\cdot )}$. If $\mathcal{E}$ actually is a quadratic function, then $\cN_1$ is a seminorm. 
\end{example}

 In the following subsections we introduce and study  interpolation functions for a given interpolation couple of functions, with respect to the $K$-method, the mean method and finally the trace method. Each of these methods are analogues, in our general setting of functions, of the classical linear interpolation methods of the same names. The interpolation function with respect to the $K$-method, introduced in Sub-section \ref{kmethod} below,  will be used  the most in the subsequent sections.

	\subsection{The $K$-method} \label{kmethod}
	
	The $K$-method is perhaps the most popular method in classical real interpolation theory of Banach spaces, and it is also the most important method for us. It is relevant for the identification of interpolation sets or, more precisely, for the identification of effective domains of interpolation functions. It admits the following variant in the context of interpolation of functions. Let $\cN = (\cN_0 ,\cN_1 )$ be an interpolation couple of functions on a Banach space $\X$, and let $\BFS$ be a Banach function space over the interval $(0,\infty )$. We first define, for given $x\in \X$, the $K$-function  
	\[
	K(x,t) := \inf \{ \cN_0 (x - v) + t \, \cN_1 (v) \st v \in \X \} \quad (t\in (0,\infty )),
	\]
	and we then define the function $\cN_{\BFS} : \X\to [0,\infty ]$ by
	\begin{align*}
	\cN_\BFS (x) := \Norm[\BFS]{ t\mapsto \frac{K(x,t)}{t} }  ,
	\end{align*}
	with the interpretation that the right hand side is $+\infty$ if the function under the norm is not in $\BFS$. We call $\cN_\BFS$ the {\em interpolation function with respect to the $K$-method} associated with the Banach function space $\BFS$. For any $\tau >0$ we also define the following ``finite'' variant $\cN_{\BFS}^{\tau} : \X\to [0,\infty )$ of this interpolation function by setting
	\begin{equation*}
	    \cN_{\BFS}^{\tau} (x)  := \Norm[\BFS]{t\mapsto \frac{K(x,t)}{t}\chi_{(0,\tau)(t)}} .
	\end{equation*}
	This interpolation function takes only values of the $K$-function for small $t$ into account. 
	
	\begin{remarks} \label{rem.k}
	(a) In the case of real interpolation between Banach spaces (see Example \ref{ex.linear}), the functions $\cN_0$ and $\cN_1$ are norms of Banach spaces $\X_0 := \dom{\cN_0}$ and $\X_1 := \dom{\cN_1}$, and it is straightforward to check that the interpolation functions $\cN_\BFS$ and $\cN_\BFS^\tau$ are norms, too. The effective domain of $\cN_\BFS$ is the classical real interpolation space defined by the $K$-method \cite[Section 3.1]{BeLo76}, \cite[Section 3.3]{BrKr91}. In the case when the Banach function space is the polynomially weighted $L^p$ space $\BFS = L^p(0,\infty ; t^{p(1-\theta) - 1}\ud t)$, then the effective domain of $\cN_\BFS$ is usually denoted by $(\X_0,\X_1)_{\theta,p}$; see also \cite[Proposition 1.13]{Lu09}, \cite[Theorem 1.8.2, p.~44]{Tr95}. 
	Note that the interpolation space $(\X_0,\X_1)_{\theta,p}$ is uniquely determined by the Banach spaces $\X_0$ and $\X_1$, as long as one considers only norms on $\X_0$ and $\X_1$ which are equivalent to $\cN_0$ and $\cN_1$, respectively. Every complete norm on $\X_0$ is equivalent to $\cN_0$ as long as the embedding of $\X_0$ into the sum space $\X = \X_0 +\X_1$ remains continuous. Of course, we assume that the topology on the sum space is fixed, and that the embeddings of $\X_i$ into $\X$ are continuous for $i=0$, $1$.
	
	Note the slightly different notation using the values $\theta$ and $p$ in the subscript instead of the space $\BFS$. Let us just mention that in Brudny \& Krugljak \cite[Section 3.3]{BrKr91} the definition of the $K$-function and the $K$-method is slightly different, the difference being the factor $\frac{1}{t}$ between the two definitions of the $K$-function.
	
	\vspace{9pt}
	(b) As a generalisation of the situation described in Remark \ref{rem.k} (a) above, one might want to define the {\em interpolation set} between the effective domains $\dom{\cN_0}$ and $\dom{\cN_1}$, associated with a Banach function space $\BFS$, as the effective domain of the interpolation function $\cN_{\BFS}$. In analogy with the classical, linear interpolation theory, a notation could be
	\[
	(\dom{\cN_0},\dom{\cN_1})_{\BFS} := \dom{\cN_{\BFS}} .
	\]
	In general, the notation of the interpolation set is however ambiguous, even for a fixed Banach function space $\BFS$. The interpolation set not only depends on the effective domains of the functions $\cN_0$ and $\cN_1$, but it really depends on the functions $\cN_0$ and $\cN_1$ themselves. In order to give an example, consider the space $\X = L^1(0,1)$, and the functions
	\[
	 \cN_0 (f) =  \Norm[L^1]{f} 
	\]
	and 
	\[
	 \cN_1 (f) = \begin{cases} 
	              \Norm[L^\infty]{f} & \text{if } f\in L^\infty (0,1) , \\[2mm]
	              +\infty & \text{otherwise. }
	             \end{cases}
	\]
Then $\cN_0$ and $\cN_1$ are norms with effective domains $\dom{\cN_0} = L^1  (0,1)$ and $\dom{\cN_1} = L^\infty (0,1)$. For $\BFS = L^p (0,\infty )$, the interpolation function $\cN_{\BFS}$ is again a norm, and its effective domain is the classical real interpolation space $(L^1 (0,1) , L^\infty (0,1))_{\frac{1}{p^*},p} =  L^p (0,1)$  \cite[Theorem 5.2.1]{BeLo76}. However, if we choose $\widehat{\cN}_0 (f) := \cN_0 (f)^\alpha$ for some $\alpha\in (0,1)$ and $\widehat{\cN}_1 (f) := \cN_1 (f)$, so that $\widehat{\cN}_0$ is a  power of the $L^1$-norm, then the effective domains of $\widehat{\cN}_i$ are equal to the effective domains of $\cN_i$, but the effective domain of the interpolation function $\widehat{\cN}_{\BFS}$ is the space $\underline{S}(( L^1(0,1),  L^\infty (0,1)  ),\overline{p}, \theta),$ one of the two spaces called {\em espaces de moyennes} (see Lions \& Peetre \cite{LiPe64}),  with $\overline{p} = ( \alpha p, p )$  and  $\theta = \frac{p - 1}{\alpha p} $ as in \cite[Section 3.12]{BeLo76}. By \cite[Theorem 3.12.1]{BeLo76}, $\underline{S}(( L^1(0,1),  L^\infty (0,1)  ),\overline{p}, \theta) = (L^1 (0,1) , L^\infty (0,1))_{\theta ,r}$ with $r = \frac{\alpha p}{1+\theta (\alpha -1)}$. This interpolation space is a Lorentz space different from $L^p (0,1).$

\vspace{9pt}

(c) A function $\cN : \X \to [0,\infty ]$ is {\em proper} if its effective domain is nonempty, that is, if it is not the constant function $\infty$. A necessary  condition for the interpolation functions $\cN_\BFS$ and $\cN_\BFS^\tau$ to be proper is the condition that there exists $x\in\X$ such that $t\mapsto \frac{K(x,t)}{t}$ (possibly multiplied with the characteristic function $\chi_{(0,\tau )}$) belongs to $\BFS$. Due to the assumption that $\BFS$ is a subset of $L^1 +L^\infty (0,\infty )$, a necessary condition for the properness of the interpolation functions therefore is the integrability of $t\mapsto \frac{K(x,t)}{t}$ on $(0,1)$ and (in case of the interpolation function $\cN_\BFS$) on $(1,\infty )$. For this, it is in turn necessary that $\liminf_{t\to 0+} K(x,t) = 0$ and (in case of the interpolation function $\cN_\BFS$) $\liminf_{t\to\infty} K(x,t) = 0$. By definition of the $K$-function, this implies $\inf \cN_0 =0$ and (in case of the interpolation function $\cN_\BFS$) $\inf \cN_1 = 0$. Without these additional assumptions on $\cN_0$ and, possibly, $\cN_1$, the interpolation functions can not be proper, no matter the choice of the Banach function space $\BFS$. 
	\end{remarks}
	
The following lemma is very much in the spirit of Remark \ref{rem.k} (c). It links the behaviour of the interpolation functions $\cN_\BFS$ and its "finite" variant $\cN_\BFS^\tau$.

	\begin{lemma}{\label{3.11}}
    Let $\cN= (\cN_0, \cN_1)$ be an interpolation couple of functions on a Banach space $\X,$ which satisfies the following properties:
    \begin{aufzi}
	       \item $\cN_0(x)$ is finite for all $x\in \X$, and
	       \item $\cN_1(v_0) =0$ at some $v_0\in \X$.
    \end{aufzi}
	Then, for every $\tau >0$, and for every Banach function space $\BFS$ over $(0,\infty )$ on which the Hardy operator is bounded,
	\begin{equation}{\label{eq 3.11.1}}
	\begin{split}
	    \cN^{\tau}_{\BFS} (x) \leq  \cN_{\BFS} (x)  \leq  \cN^{\tau}_{\BFS} (x) + \cN_0(x-v_0)\, \Norm[\BFS]{t \mapsto \frac{\chi_{(\tau,\infty)}(t)}{t}} .
	 \end{split}   
	\end{equation}
	In particular, for every $\tau>0$,
	\begin{equation*}
	    \dom{\cN^{\tau}_{\BFS}} = \dom{\cN_{\BFS}} 
	\end{equation*}
	\end{lemma}
	
	\begin{proof}
	 Let $\BFS$ be a Banach function space over $(0,\infty )$. By the Banach ideal property of $\BFS$, 
	 \[
	  \Norm[\BFS]{t \mapsto \frac{K(x,t)}{t} \chi_{(0,\tau)}(t)} \leq \Norm[\BFS]{t \mapsto \frac{K(x,t)}{t} } ,
	 \]
	 and this is just the first inequality in \eqref{eq 3.11.1}. Next, for every $t\in (0,\infty )$,
	    \begin{align*}
	       \frac{ K(x,t)}{t} &= \inf\left\{\frac{\cN_0(x-v)}{t} +  \cN_1(v) \st v\in \X \right\}\\
	       &\leq \frac{\cN_0(x-v_0)}{t}.
	    \end{align*}
	    Then, for every $\tau >0$,
	    \begin{align*}
	    \Norm[\BFS]{t \mapsto \frac{K(x,t)}{t} } &\leq \Norm[\BFS]{t \mapsto \frac{K(x,t)}{t} \chi_{(0,\tau)}(t)} + \Norm[\BFS]{t \mapsto \frac{K(x,t)}{t} \chi_{(\tau,\infty)}(t)}   \\
	    &\leq \Norm[\BFS]{t \mapsto \frac{K(x,t)}{t} \chi_{(0,\tau)}(t)} + \cN_0(x-v_0)\, \Norm[\BFS]{t \mapsto \frac{\chi_{(\tau,\infty)}(t)}{t}} ,
	    \end{align*}
    and this is the second inequality.
    
    Clearly, the first inequality in \eqref{eq 3.11.1} implies $\dom{\cN^{\tau}_{\BFS}} \subseteq \dom{\cN_{\BFS}} $. The other containment follows on noting that the Hardy operator $P$ is bounded on $E$ and $P( \chi_{(0,\tau)})(t) = \chi_{(0,\tau)}(t) +  \tau \frac{\chi_{(\tau,\infty)}(t)}{t}$.
    
	\end{proof}

	\subsection{The mean method} 
	
	Let $\cN = (\cN_0 ,\cN_1 )$ be an interpolation couple of functions on a Banach space $\X$, and let $\BBFS = (\BFS_0, \BFS_1)$ be a couple of Banach function spaces over the interval $(0,\infty )$.  Further, let $G$ be any subspace of $L^0 (0,\infty ;\X )$. 
	We then define a new function $\cN_{\BBFS}^G : \X\to [0,\infty ]$ by
	\begin{align*}
	\cN_\BBFS^G (x) := \inf \{ \Norm[\BFS_0]{ t\mapsto \frac{\cN_0 (x - u (t ))}{t}} + & \Norm[\BFS_1]{ t\mapsto \cN_1 (u (t )) } \st u \in G \}.
	\end{align*}
	We call $\cN_\BBFS^G$ the {\em interpolation function with respect to the mean method}, associated with the couple $\BBFS = (\BFS_0 ,\BFS_1 )$ and associated with $G$. Typical examples of spaces $G$  in the following are the full space of measurable functions ( we  write $\cN_{\BBFS}^{L^0}$ in short for the interpolation function) and the space of continuous functions $G = C ([0,\infty ) ;\X )$ (we  write $\cN_{\BBFS}^{C}$ for the resulting interpolation function). We further define a ``finite'' version of these interpolation functions by setting, for every $\tau \in (0,\infty )$,
	\begin{align*}
	\cN_\BBFS^{G,\tau} (x) := \inf \{ \Norm[\BFS_0]{ t\mapsto \frac{\cN_0 (x - u (t ))}{t}\, \chi_{(0,\tau )} (t) } + & \Norm[\BFS_1]{ t\mapsto \cN_1 (u (t )) \, \chi_{(0,\tau )} (t) } \st u \in G \}.
	\end{align*}
	Clearly, for every $x\in\X$ and for every $\tau\in (0,\infty )$, 
    \begin{align*}
       & \cN_{\BBFS}^{L^0} (x) \leq \cN_{\BBFS}^{C}(x) \text{ and} \\ 
       & \cN_{\BBFS}^{L^0,\tau} (x) \leq \cN_{\BBFS}^{C,\tau} (x) .
    \end{align*}
	
	\begin{remarks} 
	(a) Again, in the special case when $\cN_0$ and $\cN_1$ are norms, then the interpolation functions $\cN_{\BBFS}^{G}$ and $\cN_{\BBFS}^{G,\tau}$ are norms, too. The effective domain of $\cN_{\BBFS}^{L^0}$ is the classical interpolation space defined by the mean method; its definition goes back to Lions and Peetre \cite{LiPe64}.
	\vspace{9pt}
	
	(b) As in the case of the $K$-method, one might want to define the {\em interpolation set with respect to the mean method} between the effective domains $\dom{\cN_0}$ and $\dom{\cN_1}$, associated with the Banach function space $\BFS$, as the effective domain of the interpolation function $\cN_{\BBFS}^G$: 
	\[
	(\dom{\cN_0},\dom{\cN_1})_{\BBFS}^G := \dom{\cN_{\BBFS}^G} .
	\]
	The same remark as in Remark \ref{rem.k} (b) applies about the ambiguity of the interpolation set, which does not depend just on $\dom{\cN_0}$ and $\dom{\cN_1}$, but also on the functions $\cN_0$ and $\cN_1.$  One may repeat the example from Remark \ref{rem.k} (b), or alternatively, use  Lemma \ref{lm 3.10} below,   stating the equivalence of the mean method and the $K$-method. 
\vspace{9pt}

(c)  A necessary condition for $\cN_\BFS^C$ (or $\cN_\BFS^{C,\tau}$) to be finite is the condition that the norm 
\[
 \Norm[\BFS_0]{ t\mapsto \frac{\cN_0 (x - w (t))}{t} \chi_{(0,\tau )} (t) }
\]
is finite. If $\cN_0$ is a norm on a Banach space $\X_0\subseteq\X$, then by Lemma \ref{lem.ac} this norm is for example finite if $u\in\mathring{AC} ([0,\infty );\X )$. 
\end{remarks}

If $\BFS_0 = \BFS_1 = \BFS$, then we can compare the mean method with the $K$-method.  In the following lemma, we identify a single Banach function space $\BFS$ over the interval $(0,\infty )$ with the couple $\BBFS = (\BFS , \BFS)$, that is, we put $\BFS_0 := \BFS_1 := \BFS$ in the mean method. 
	
	\begin{lemma}[Equivalence of $K$-method and mean method] \label{lem.k-mean} {\label{lm 3.10}}
		Let $\cN = (\cN_0,\cN_1 )$ be an interpolation couple of functions on a Banach space $\X$, and let $\BFS$ be a Banach function space over the interval $(0,\infty )$. Then, for every $\tau\in (0,\infty )$,
		\begin{align*}
		& \cN_\BFS \leq \cN_\BBFS^{L^0} \leq 2 \, \cN_\BFS \text{ and} \\
		& \cN_\BFS^\tau \leq \cN_\BBFS^{L^0,\tau} \leq 2 \, \cN_\BFS^\tau
		\end{align*}
		and in particular,
		\begin{align*}
		& \dom{\cN_{\BFS}} = \dom{\cN_{\BBFS}^{L^0}} \text{ and} \\
		& \dom{\cN_{\BFS}^\tau} = \dom{\cN_{\BBFS}^{L^0,\tau}}.
		\end{align*}
	\end{lemma}
	
	\begin{proof}
    Let $x\in \X$, and let $\varepsilon >0$ be given. Let $t_n := (1+\varepsilon )^{-n}$ ($n\in\Z$), so that $t_0 = 1$, $t_n > t_{n+1}$ for every $n$, and $\lim_{n\to\infty} t_n = 0$, $\lim_{n\to -\infty} t_n = \infty$. For every $v \in \X$ and for every $t\in (t_{n+1},t_n]$ one has
    \begin{align*}
    K(x,t_{n+1}) & \leq \cN_0 (x - v) + t_{n+1} \, \cN_1 (v) \\
    & \leq \cN_0 (x - v) + t \, \cN_1 (v) .
    \end{align*}
    Taking the infimum on the right hand side of this inequality yields
    \[
    K(x,t_{n+1}) \leq K(x,t) .
    \]
    Similarly, for every $t\in (t_{n+1} ,t_n]$,
    \begin{align*}
    K(x,t) & \leq \cN_0 (x - v) + t \, \cN_1 (v) \\
    & \leq (1+\varepsilon) \, (\cN_0 (x - v) + t_{n+1} \, \cN_1 (v) ) ,
    \end{align*}
    and therefore 
    \[
    K(x,t ) \leq (1+\varepsilon) \, K(x,t_{n+1}) .
    \]
    Now, for every $n\in\Z$ there exists a $v_n \in \X$ such that
    \[
    K (x,t_n) \leq \cN_0 (x - v_n) + t_n \, \cN_1 (v_n ) \leq K(x,t_n) + \varepsilon \, \min \{ t_n , 1/t_n \} .
    \]
    Set
    \[
    u (t) = v_n \text{ for } t \in (t_{n+1} , t_n] .
    \]
    Then $u : (0,\infty )\to \X$ is measurable and for every $t \in (t_{n+1},t_n]$,
    \begin{align*}
    \cN_0 (x - u (t)) = \cN_0 (x - v_n) & \leq K(x,t_n) + \varepsilon \, \min \{ t_n ,1/t_n \} \\
    & \leq (1 + \varepsilon ) ( K(x,t) + \varepsilon \, \min \{ t ,1/t \} ) , 
    \end{align*}
    and similarly
    \begin{align*}
    \cN_1 (u (t)) = \cN_1 (v_n) & \leq \frac{K(x,t_n)}{t_n} + \varepsilon \, \min \{ 1 , 1/t_n^2 \} \\
    & \leq (1 + \varepsilon ) ( \frac{K(x,t)}{t} + \varepsilon \, \min \{ 1, 1/t^2\} ) .
    \end{align*}
    Hence,
    \[
    \Norm[\BFS]{ t\mapsto \frac{\cN_0 (x - u (t))}{t} } \leq (1+\varepsilon ) \, \left( \Norm[\BFS]{ t\mapsto \frac{K(x,t)}{t} } + \varepsilon \, \Norm[\BFS]{ t\mapsto \min \{ 1 ,1/t^2\}  } \right)
    \]
    and 
    \[
    \Norm[\BFS]{ t\mapsto \cN_1 (u (t)) } \leq (1+\varepsilon ) \, \left( \Norm[\BFS]{ t\mapsto \frac{K(x,t)}{t} } + \varepsilon \, \Norm[\BFS]{ t\mapsto \min \{ 1 ,1/t^2\}  } \right) ,
    \]
    so that 
    \begin{align*}
    \cN_\BBFS^{L^0} (x) & \leq \Norm[\BFS]{ t\mapsto \frac{\cN_0 (x - u (t))}{t} } +  \Norm[\BFS]{ t\mapsto \cN_1 (u (t)) } \\
    & \leq 2\, (1+\varepsilon ) \, (\cN_\BFS (x) +  \varepsilon \, \Norm[\BFS]{ t\mapsto \min \{ 1 ,1/t^2 \}  } ) .
    \end{align*}
	The function $t\mapsto \min \{ 1 ,1/t^2 \}$ belongs to $L^1 \cap L^\infty (0,\infty )$, and therefore also to $\BFS$. Its norm in $\BFS$ is therefore finite. Since $\varepsilon >0$ is arbitrary, this  proves one inequality. For the other inequality, note that for every $\varepsilon >0$ there exists a measurable function $u : [0,\infty )\to \X$ such that
    \[
    \cN_\BBFS^{L^0} (x) \leq \Norm[\BFS]{ t\mapsto \frac{\cN_0 (x - u(t))}{t} } + \Norm[\BFS]{ t\mapsto \cN_1 (u (t)) } \leq \cN_\BBFS^{L^0} (x) + \varepsilon .
    \]
    Clearly, by definition of the $K$-functional, for every $t\in [0,\infty )$,
    \[
    \frac{K(x,t)}{t} \leq \frac{\cN_0 (x - u(t))}{t} + \cN_1 (u(t)) , 
    \]
    so that, by the triangle inequality,
    \begin{align*}
    \cN_\BFS (x) & = \Norm[\BFS]{ t\mapsto \frac{K(x,t)}{t} } \\
    & \leq \Norm[\BFS]{ t\mapsto \frac{\cN_0 (x - u(t))}{t} } + \Norm[\BFS]{ t\mapsto \cN_1 (u(t)) } \\
    & \leq \cN_\BBFS^{L^0} (x) + \varepsilon .
    \end{align*}
    Since $\varepsilon >0$ is arbitrary, this implies the second inequality. The corresponding statements for the ``finite'' interpolation functions are proved as above; it suffices to multiply the pointwise estimates by the characteristic function $\chi_{(0,\tau )}$ before taking the norm in $\BFS$.
	\end{proof}

The following theorem is a justification to consider domains of interpolation functions as {\em interpolation sets}. 

\begin{theorem}[Interpolation theorem] \label{lem.mean1}
    Let $\X$ and $\Y$ be two Banach spaces, $\cN = (\cN_0 ,\cN_1)$ an interpolation couple of functions on $\X$, $\cM = (\cM_0 ,\cM_1 )$ an interpolation couple of functions on $\Y$, and let $T : \X \to \Y$ be a measurable function. Suppose that there exist constants $L\geq 0$, $a>0$ and an increasing function $b : \R^+ \to \R^+$ such that, for every $x$, $y\in\X$,
    \begin{equation} \label{eq.mean1.1}
    \cM_0 (Tx -Ty) \leq L\, \cN_0 (x-y) 
    \end{equation}
    and
    \begin{equation} \label{eq.mean1.2}
    \cM_1 (Tx) \le a\, \cN_1(x) + b(\Norm[\X]{x}) .
    \end{equation}
    Then, for every pair $\BBFS = (\BFS_0 , \BFS_1)$ of Banach function spaces over the interval $(0,\infty )$ and for every $\tau\in (0,\infty )$, there exist a constant $\tilde{a}\geq 0$ and a finite valued function $\tilde{b} : \X \to \R^+$ such that
    \[
    \cM_\BBFS^{C,\tau} (Tx) \le \tilde{a}\, \cN_\BBFS^{C,\tau} (x) + \tilde{b} (x) \text{ for all } x \in \X.
    \]
    In particular
    \[
    T( \dom{\cN_\BBFS^{C,\tau}} ) \subseteq \dom{\cM_\BBFS^{C,\tau}}.
    \]
    If, in condition \eqref{eq.mean1.2}, $b=0$, then one can choose $\tilde{b} = 0$ in the conclusion and replace the finite interpolation functions by $\cM_\BFS^C$ and $\cN_\BFS^C$. In particular, if $b=0$ in condition \eqref{eq.mean1.2}, then
    \[
    T( \dom{\cN_\BFS^{C}} ) \subseteq \dom{\cM_\BFS^{C}}.
    \]
\end{theorem}

\begin{remarks}
(a) If $(\X_0,\X_1)$ and $(\Y_0,\Y_1)$ are interpolation couples of Banach spaces, if $\X = \X_0 + \X_1$ and $\Y = \Y_0 + \Y_1$, if $\cN_0$ and $\cN_1$ are the norms on $\X_0$ and $\X_1$, and $\cM_0$ and $\cM_1$ are the norms on $\Y_0$ and $\Y_1$, respectively, then we are in the situation of classical interpolation theory. The conditions \eqref{eq.mean1.1} and \eqref{eq.mean1.2} (with $b=0$) for a {\em linear} operator $T : \X \to \Y$ mean that $T$ maps $\X_i$ boundedly into $\Y_i$ for both $i=0$ and $i=1$. Theorem \ref{lem.mean1} implies that $T$ maps the domain of $\cN_\BBFS^{C}$ into the domain of $\cM_\BBFS^{C}$, and therefore, by the very definition of interpolation spaces, these domains {\em are} interpolation spaces. Theorem \ref{lem.mean1} therefore recovers many classical interpolation theorems: for the {\em espaces de moyennes} in which $\BFS_0$ and $\BFS_1$ are two polynomially weighted $L^p$-spaces, see for example Lions \& Peetre \cite{LiPe64}; see also \cite[Theorem 3.1.2]{BeLo76}, \cite{BeSh88}, \cite{BrKr91} and many other instances for the special case $\BFS_0 = \BFS_1$.
\vspace{9pt}

(b) Despite the fact that the classical interpolation spaces are designed for the interpolation of linear operators, it is known that nonlinear operators may interpolate, too. For nonlinear operators which are Lipschitz continuous from $\X_i$ into $\Y_i$ for both $i=0$ and $i=1$ we refer to \cite{Or54,LoSh68,LoSh69}. Theorem \ref{lem.mean1}, however, has slightly different conditions and does not recover the interpolation theorem for Lipschitz operators: in Theorem \ref{lem.mean1}, 
the operator $T$ is required to be Lipschitz continuous from $\X_0$ into $\Y_0$, but it is only bounded from $\X_1$ into $\Y_1$ (if $b=0$). Actually, if $b=0$, then by \cite[Example 4.1.5 (b)]{BrKr91}, $T$ is a Gagliardo-Peetre operator (see \cite[Definition 4.1.1]{BrKr91} for the definition of Gagliardo-Peetre operators), and hence, by \cite[Proposition 4.1.3]{BrKr91}, $K(Tx,t,\cM ) \leq C\, K(x,t,\cN )$ for some constant $C\geq 0$ and for all $x\in \X_0 +\X_1$ and all $t>0$. This implies that for every Banach function space $\BFS$, $\cM_\BFS (Tx) \leq C\, \cN_\BFS (x)$, which is a version of Theorem \ref{lem.mean1} for the interpolation functions associated with the $K$-method or, by Lemma \ref{lem.k-mean}, for the interpolation functions $\cN_\BBFS^{L^0}$ and $\cM_\BBFS^{L^0}$. Theorem \ref{lem.mean1} therefore partially recovers some findings of \cite[Section 4.1]{BrKr91}. If $b\not= 0$, then in general we have to consider domains of the "finite" interpolation functions $\cN_\BBFS^{C,\tau}$ and $\cM_\BBFS^{C,\tau}$ instead of the classical real interpolation spaces. 
\vspace{9pt}

(c) Still, if $(\X_0,\X_1)$ and $(\Y_0,\Y_1)$ are interpolation couples of Banach spaces as above, and if $\cM_0$ and $\cM_1$ are the norms on $\Y_0$ and $\Y_1$, respectively, but $\cN_0 = \norm[\X_0]{\cdot}^{\alpha_0}$ and $\cN_1 = \norm[\X_1]{\cdot}^{\alpha_1}$ are powers of the norms in $\X_0$ and $\X_1$ for some exponents $\alpha_0$, $\alpha_1\in [0,1]$, the condition \eqref{eq.mean1.1} reads as
\[
 \norm[\Y_0]{Tx -Ty} \leq L\, \norm[\X_0]{x-y}^{\alpha_0} ,
\]
which is H\"older continuity of the operator $T$ from $\X_0$ into $\Y_0$. If one replaced the condition \eqref{eq.mean1.2} by the H\"older condition
\[
 \norm[\Y_1]{Tx -Ty} \leq L\, \norm[\X_1]{x-y}^{\alpha_1} ,
\]
then one would be in the interpolation setting of Peetre \cite{Pe70} and Tartar \cite{Tar72} (see also Coulhon \& Hauer \cite[Theorem 4.6]{CoHa18}). Such H\"older continuous operators interpolate to the so-called {\em nonlinear interpolation spaces} defined in these articles; see also Example \ref{ex.nonlinear}.  These nonlinear interpolation spaces actually coincide with the effective domains of our interpolation functions $\cN_\BFS$ (and the $L$-function defined in \cite{Pe70} coincides with our $K$-function associated with the powers of norms). Theorem \ref{lem.mean1} does not recover these interpolation theorems; we repeat that the condition \eqref{eq.mean1.2} is a kind of boundedness condition of $T$ with respect to the functions $\cN_1$ and $\cM_1$. 
\end{remarks}
\vspace{9pt}
\begin{proof}[Proof of Theorem \ref{lem.mean1}]
Fix $x\in \X$. If $\cN_{\BBFS}^{C,\tau}(x) =\infty$, then the claimed inequality is trivial, whatever the value of $\tilde{b} (x)$. We may put $\tilde{b}(x) = 0$ in this case. Assume now $\cN_{\BBFS}^{C,\tau}(x) < \infty$. For every continuous function $u : [0,\infty ) \to \X$ with $u(0)=x$ we get
    \begin{align*}
    \cM_\BBFS^{C,\tau} (Tx) &\le \Norm[\BFS_0]{ t\mapsto \frac{ \cM_0 (T x - T u (t)) }{t} \, \chi_{(0,\tau )} (t)} + \Norm[\BFS_1]{ t\mapsto\cM_1 (T u (t )) \, \chi_{(0,\tau )} (t)} \\
    & \le L \, \Norm[\BFS_0]{ t\mapsto \frac{ \cN_0 (x - u (t)) }{t} \, \chi_{(0,\tau )} (t)} + a \, \Norm[\BFS_1]{ t\mapsto \cN_1 (u (t )) \, \chi_{(0,\tau )} (t)} \\
    & \phantom{\leq\ } + \Norm[\BFS_1]{ t \mapsto b(\Norm[\X]{u(t)}) \, \chi_{(0,\tau )} (t)}.
    \end{align*}
Set $\tilde{a} = \max\{ L, a \}$. There exists $u\in C([0,\infty );\X)$ such that 
    \begin{equation*}
    \Norm[\BFS_0]{ t\mapsto \frac{ \cN_0 (x - u (t)) }{t} \, \chi_{(0,\tau )} (t)} +  \Norm[\BFS_1]{ t\mapsto \cN_1 (u (t )) \, \chi_{(0,\tau )} (t) } \leq \cN_{\BBFS}^{C,\tau}(x) + 1.
    \end{equation*}
    By continuity of $u$ and since $b$ is increasing, the function $t \mapsto  b(\Norm[\X]{u(t)})$ is locally bounded and therefore 
    \[
    t\mapsto b(\Norm[\X]{u(t)})\, \chi_{(0,\tau )} (t) \in L^1 \cap L^\infty (0,\infty ) \subseteq \BFS_1.
    \]
    Putting $\tilde{b}(x)= \Norm[\BFS_1]{t \mapsto  b(\Norm[\X]{u(t)}) \, \chi_{(0,\tau )} (t)}$ and taking the infimum leads to
		\[
		\cM_\BBFS^{C,\tau} (Tx) \le \tilde{a} \cN_\BBFS^{C,\tau} (x) + \tilde{b} ( x).
		\]
		Since $\tilde{b}$ takes only finite values, this implies $T( \dom{\cN_\BBFS^{C,\tau}} ) \subseteq \dom{\cM_\BBFS^{C,\tau}}$.
		
	If $b=0$, then one can take $\tilde{b} = 0$ and repeat the above calculations without the characteristic functions $\chi_{(0,\tau )}$ under the norms, and one obtains a corresponding inequality for the interpolation functions $\cN_\BBFS^C$ and $\cM_\BBFS^C$, and a corresponding inclusion relation for their effective domains.
	\end{proof}

	\subsection{The trace method}
	
	The following is a variant of the mean method and may be more relevant for the applications to Cauchy problems. Let $\cN = (\cN_0 ,\cN_1 )$ be an interpolation couple of functions on a Banach space $\X$, let $\BBFS = (\BFS_0, \BFS_1)$ be a couple of Banach function spaces over the interval $(0,\infty )$, and let $\tau\in (0,\infty )$. 
	We then define two new functions $\cN_{\BBFS}^{tr}$, $\cN_{\BBFS}^{tr,\tau} : \X\to [0,\infty ]$ by
	\begin{align*}
	\cN_\BBFS^{tr} (x) := \inf \{ \Norm[\BFS_0]{ t\mapsto \cN_0 (\dot w (t))}  + &  \Norm[\BFS_1]{ t\mapsto \cN_1 ( w(t))} \st \\
	& w\in W^{1,1}_{loc} ([0,\infty );\X) \text{ and } w(0) = x \} ,
	\end{align*}
	and
	\begin{align*}
	\cN_\BBFS^{tr,\tau} (x) := \inf \{ \Norm[\BFS_0]{ t\mapsto \cN_0 (\dot w (t)) \chi_{(0,\tau )} (t)}  + &  \Norm[\BFS_1]{ t\mapsto \cN_1 ( w(t)) \chi_{(0,\tau )} (t)} \st \\
	& w\in W^{1,1}_{loc} ([0,\infty );\X) \text{ and } w(0) = x \} .
	\end{align*}	
	We call $\cN_\BBFS^{tr}$ and $\cN_\BBFS^{tr,\tau}$ the {\em interpolation function (resp. finite interpolation function) with respect to the trace method} associated with the couple $\BBFS = (\BFS_0 ,\BFS_1 )$. In the following lemma, we use the interpolation functions $\cN_\BBFS^{W^{1,1}}$ and $\cN_\BBFS^{W^{1,1},\tau}$ with respect to the mean method and associated with $G = W^{1,1}_{loc} (0,\infty ;\X )$. 
	
	\begin{lemma}[Comparison of trace method and mean method] \label{lem.trace.mean}
    Let $\cN = (\cN_0,\cN_1 )$ be an interpolation couple of functions on a Banach space $\X$, and let $\BBFS = (\BFS_0 , \BFS_1)$ be a couple of Banach function spaces over the interval $(0,\infty )$. 
    If $\cN_0$ is a norm, and if the Hardy operator $P$ is bounded in $\BFS_0$, then, for every $x\in \X$,
    \begin{align*}
    \cN_\BBFS^{W^{1,1}} (x) & \leq \Norm[\cL (\BFS_0 )]{P} \, \cN^{tr}_\BBFS (x) \text{ and} \\
    \cN_\BBFS^{W^{1,1},\tau} (x) & \leq \Norm[\cL (\BFS_0 )]{P} \, \cN^{tr,\tau}_\BBFS (x) .
    \end{align*}
	\end{lemma}
	
	\begin{proof} 
    Let $x\in \X$ and $w\in W^{1,1}_{loc} (0,\infty ;\X )$ be such that $w(0) = x$. Then, for every $t\in (0,\infty )$, 
    \[
     \frac{x-w(t)}{t} = - \frac{1}{t} \int_0^t \dot w(s) \,\ud s ,
    \]
    and therefore, since $\cN_0$ is a norm,
    \[
     \frac{\cN_0 (x-w(t))}{t} \leq \frac{1}{t} \int_0^t \cN_0 (\dot w(s)) \,\ud s .
    \]
    Hence, 
    \[
    \Norm[\BFS_0]{ t\mapsto \frac{\cN_0 (x-w(t))}{t}  } \leq \Norm[\BFS_0]{ P \cN_0 (\dot w (\cdot )) } \leq \Norm[\cL (\BFS_0 )]{P} \, \Norm[\BFS_0]{ \cN_0 (\dot w (\cdot )) } , 
    \]
    and, since $\Norm[\cL (\BFS_0 )]{P} \geq 1$, 
    \[
    \Norm[\BFS_0]{ t\mapsto \frac{\cN_0 (x-w(t))}{t}  } + \Norm[\BFS_1]{ \cN_1 (w(\cdot ))  } \leq \Norm[\cL (\BFS_0 )]{P} \, (  \Norm[\BFS_0]{ \cN_0 (\dot w (\cdot )) } +  \Norm[\BFS_1]{ \cN_1 (w(\cdot ))  } ) .
    \]
    Taking the infimum over all $w\in {W}^{1,1}_{loc} (0,\infty ;\X )$  implies the first inequality. The second inequality is proved in the same way, but inserting the characteristic function $\chi_{(0,\tau )}$ before taking the norms in $\BFS_0$ and $\BFS_1$.
	\end{proof}

	\section{Interpolation associated with $m$-accretive operators}
	
	Let $\X$ be a Banach space. A relation $A\subseteq \X\times \X$ (we use in the following the word {\em operator} instead of {\em relation}) is {\em accretive of type $\omega\in\R$} if for every $(x,f)$, $(\hat{x},\hat{f} )\in A$ and for every $\lambda >0$ with $\lambda \omega < 1$,
	\[
	\norm[\X]{x-\hat{x} + \lambda (f-\hat{f})} \geq  \,(1-\lambda \omega) \norm[\X]{x-\hat{x}} . 
	\]
	It follows from this very definition that for every $\lambda >0$ with $\lambda \omega < 1$, the operator $I+\lambda A$ is injective. We say that the operator $A$ is {\em $m$-accretive of type $\omega$} if it is accretive of type $\omega$ and if $I+\lambda A$ is surjective onto $\X$ (hence, bijective from $\dom{A}$ to $\X$) for every $\lambda >0$ with $\lambda \omega < 1$, and in this case, we write $J_\lambda^A := (I+\lambda A)^{-1}$ for the {\em resolvent} of $A$ (if $A$ is clear from the context, we simply write $J_\lambda$). The resolvent $J_\lambda$ is single-valued and Lipschitz continuous on $\X$ with Lipschitz constant $\frac{1}{1-\lambda\omega}$. By the Crandall-Liggett theorem (\cite[Theorem 4.3]{Bar10}, \cite[Chapter 3]{BeCrPa99}), if $A$ is $m$-accretive of type $\omega$, then the Cauchy problem 
	\begin{equation}\tag{ACP} \label{acp}
\begin{aligned}
\dot{u} + A u &\ni 0 , \\
u(0) &= x \nonumber
\end{aligned}
\end{equation}
admits, for every $x \in \overline{\dom{A}}$, a unique mild solution $u \in C(\mathbb{R}_{+}, \X)$. Mild solutions arise for example as locally uniform limits of solutions of associated implicit Euler schemes, and they are in general only continuous. Setting, for each $ x \in \overline{\dom{A}} $ and for each $t\geq 0$, $ S(t) x := u(t)$, where $u$ is the solution of the above Cauchy problem for the initial value $x$, we obtain a (nonlinear) strongly continuous semigroup  $(S(t))_{t \geq 0}$ of Lipschitz continuous operators on $\overline{\dom{A}}.$ We say that $A$  generates (or, in a more common terminology, $-A$ generates) a strongly continuous semigroup $S$ on $\overline{\dom{A}}$. 
	
In this section, we interpolate between the norm on $\X$ and the set norm of $A$, that is, in terms of the previous section, we consider the interpolation couple $\cN = (\cN_0 ,\cN_1 )$ where $\cN_0 = \norm[\X]{\cdot}$ is the norm on $\X$ and the function $\cN_1 (x) = |Ax|$, where  $ | \cdot | $ represents the set norm in $\X$,  whose effective domain is just the domain of the operator $A$.
Note that with these choices, we have for $x \in \X$, and $ \BFS$ a Banach function space over the interval $(0,\infty)$ and for every $\tau>0$, 
\begin{align*}
K(x,t) &=  \inf \{  \norm[\X]{x - v }+ t | A v | \st v \in \dom{A} \} \quad (t\in (0,\infty )) ,\\
\dom{\cN_\BFS}  & = \left\{ x \in \X \st t \mapsto \frac{K(x,t)}{t} \in \BFS \right\} , \\
\dom{\cN_\BFS^{\tau}}  & = \left\{ x \in \X \st t \mapsto \frac{K(x,t)}{t}\chi_{(0,\tau)}(t)  \in \BFS \right\}.
\end{align*}
We estimate the interpolation function $\cN_\BFS^{\tau}$ (and in some statements also $\cN_\BFS$) from above and below by functions which involve both the resolvent of $A$ and the semigroup generated by $-A$. Recall for that aim that every mild solution $u$ of \eqref{acp} is in turn an {\em integral solution}, that is, the following holds: for every $s$, $t \in [0,\infty)$ with $s < t$, and for all $ (\hat{x}, \hat{f}) \in A$, 
\begin{align}\label{eq 4.0.1}
\norm[\X]{ u(t) - \hat{x}}&\leq \norm[\X]{ u(s) - \hat{x}} +
 \int_{s}^{t} \left [ u(\tau) - \hat{x},  - \hat{f} \right ]\,\ud s \tau + \omega \int_{s}^{t} \norm[\X]{ u(\tau) - \hat{x}} \,\ud \tau ,
\end{align}
 where for any $x$, $y \in \X,$ 
\[
 [x,y] := \lim_{\lambda \rightarrow 0^{+}}\frac{1}{\lambda}(\norm[\X]{ x+\lambda y })- \norm[\X]{ x} ) = \inf_{\lambda >0}\frac{1}{\lambda}(\norm[\X]{ x+\lambda y}  )- \norm[\X]{ x} )
 \]
 is the {\em Kato bracket}.  Note that it follows from the definition of an integral solution and Gronwall's Lemma that for every $x\in \dom{A}$ the solution of the abstract Cauchy problem \eqref{acp} is locally Lipschitz continuous. More precisely, for every $x\in \dom{A}$ and for every $s$, $t \geq 0$ with $s\leq t$ we have 
 \begin{equation}{\label{eq 4.0.2 }}
   \Norm[\X]{S(t)x - S(s)x} \leq e^{\omega t } \,|t-s| \,|Ax|.  
 \end{equation}
 
\begin{theorem} \label{th4.1}
Let $A\subseteq \X \times \X$ be an $m$-accretive operator of type $\omega \geq 0$ on a Banach space $\X$, and let $S$ be the semigroup generated by $-A$. Then, for every $x \in \overline{\dom A}$ and every $t>0$ with $t\omega <1$,
\begin{aufzi}
\item\label{eq 4.1.1}{ \begin{equation*}
    \frac12 \, \frac{K(x, t)}{t}  \leq  \frac{\norm[\X]{x -J_tx }}{t} \leq \frac{2-t\omega}{1-t\omega} \frac{K(x,t)}{t}, 
    \end{equation*}}
\item \label{eq 4.1.2}{\begin{equation*}
(1-t\omega ) \, \frac{\ud}{\ud t} \Var_{J_{\cdot} x} (t )  \leq  \frac{\norm[\X]{x -J_tx }}{t} \leq  \frac{1}{t} \int_0^t \frac{\ud}{\ud s} \Var_{J_{\cdot} x} (s) \;\ud s,
\end{equation*}} 
\item\label{eq 4.1.3} {\begin{equation*}
\phantom{\leq } \frac{\norm[\X]{x - S(t)x}}{t} \leq (1+e^{\omega t})  \frac{K(x, t )}{t} , \end{equation*}}
\item\label{eq 4.1.4}{\begin{equation*} 
\norm[\X]{x-J_{t}x } \leq  \frac{3 - t \omega + e^{\omega t}}{1-t\omega} \, \frac{1}{t} \int_0^{t} \norm[\X]{ x-S(s)x} \,\ud s .
\end{equation*}}
\end{aufzi}
\end{theorem}

\begin{proof} 
Let $x \in   \overline{\dom{A}}$ and $t>0$ such that $t\omega <1$. Let $A_t = \frac{I - J_t}{t}$ be the Yosida approximation of $A$. Note that, for every $x\in\X$,  $J_t x \in \dom{A}$ and that $A_tx \in AJ_t x$. Therefore, by definition of the $K$-function and by the definition of the set norm,   
\begin{align*}
K(x, t) & \leq \norm[\X]{x - J_tx}  + t \, |AJ_tx| \\
& \leq \norm[\X]{x - J_tx}  + t\, \norm[\X]{A_tx} \\
& = 2 \norm[\X]{x - J_tx} .
\end{align*}
This proves the first inequality in \ref{eq 4.1.1}. 

The second inequality in \ref{eq 4.1.1} and the inequality \ref{eq 4.1.3} are proved in a similar way. First, by the Lipschitz continuity of $J_t$ on $\X$, for every $v\in \dom A$ we have
\begin{align*} 
\norm[\X]{x -J_t x }  & \leq \norm[\X]{ (x -J_t x) - (v -J_t v) } + \norm[\X]{ v -J_t v } \nonumber \\
& \leq  \norm[\X]{ x-v } + \norm[\X]{ J_t x-J_t v} + t\, \norm[\X]{ A_tv} \nonumber \\
& \leq \norm[\X]{ x-v } + \frac{1}{1-t \omega} \, \norm[\X]{ x-v }  + t\, \norm[\X]{A_tv}  . 
\end{align*}
Second, for the Yosida approximation we have on using \cite[Prop. 3.2 (d)]{Bar10}, the estimate
\[
 \norm[\X]{A_t v} \leq  \frac{1}{1-t\omega} \, \vert Av\vert .
\]
Hence, for every $v\in\dom{A}$,
\begin{equation}\label{eq 4.1.5}
\norm[\X]{ x -J_tx } \leq \frac{2-t\omega}{1-t\omega}  \norm[\X]{ x-v} + \frac{t}{1-t\omega} \, \vert Av \vert.
\end{equation}
Taking the infimum over $v\in\dom{A}$ on the right hand side of this inequality and noting that $1\leq 2-t\omega$ for every $t>0$ with $t \omega <1$, we find
\[
\norm[\X]{ x - J_tx }  \leq \frac{2-t\omega}{1-t\omega} \, K(x,t) .
\]
This proves the second inequality in \ref{eq 4.1.1}. 

Next, by the Lipschitz continuity of $S(t)$ on $\overline{\dom{A}}$ (with Lipschitz constant $e^{\omega t}$), for every $v\in \dom{A}$ we have
\begin{align*} 
\norm[\X]{ x -S(t) x}   & \leq \norm[\X]{ (x -S(t) x) - (v -S(t) v) } + \norm[\X]{ v -S(t) v } \nonumber \\
& = \norm[\X]{ x-v } + \norm[\X]{ S(t) x- S(t) v}  + t\, \norm[\X]{ \frac{v - S(t)v}{t} } \nonumber \\
& \leq (1 + e^{\omega t} )\, \norm[\X]{ x-v }  + t\, \norm[\X]{ \frac{v - S(t)v}{t} } . \nonumber
\end{align*}
Now, as remarked in \eqref{eq 4.0.2 }, for every $v\in\dom{A}$ the trajectory $t\mapsto S(t)v$ is Lipschitz continuous in time. In particular (see \eqref{eq 4.0.2 }),
\[
 \norm[\X]{ \frac{v - S(t)v}{t} } \leq e^{\omega t} \, |Av | \text{ for every } t>0 . 
\]
Hence, for every $v\in\dom{A}$,
\[
\norm[\X]{ x -S(t)x } \leq (1 + e^{\omega t} ) \,  \norm[\X]{ x-v } + t\, e^{\omega t}\, \vert Av \vert .
\]
Taking the infimum over $v\in\dom{A}$ on the right hand side of this inequality yields
\[
\norm[\X]{ x - S(t) x} \leq (1 +e^{\omega t} ) \, K(x,t) ,
\]
which is the inequality \ref{eq 4.1.3}.

Next, by the resolvent identity, for every $\mu$, $\lambda >0$ with $\mu\tau$, $\lambda\tau <1$, 
\[
J_\lambda x - J_\mu x = J_\lambda x - J_\lambda ( \frac{\lambda}{\mu} x + (1-\frac{\lambda}{\mu}) J_\mu x) . 
\]
The Lipschitz continuity of the resolvent implies 
\begin{equation} \label{eq.res.lip}
\begin{split}
\norm[\X]{ J_\lambda x - J_\mu x } & \leq \frac{1}{1-\lambda \omega} \, \norm[\X]{ x -  ( \frac{\lambda}{\mu} x + (1-\frac{\lambda}{\mu}) J_\mu x ) } \\
& =  \frac{ |\lambda - \mu |}{1-\lambda \omega} \, \norm[\X]{ \frac{x-J_\mu x}{\mu} } 
\end{split}
\end{equation}
From here it follows that the variation of the function $ \lambda \mapsto J_{\lambda} x $ is Lipschitz continuous and being  scalar-valued, is in turn differentiable a.e., so that,   
\[
\frac{\ud}{\ud\lambda} \Var_{J_{\cdot} x} (\lambda ) \leq \, \frac{1}{1-\lambda \omega}\,\norm[\X]{ \frac{x-J_\lambda x}{\lambda} } ,
\]
which is the first inequality in \ref{eq 4.1.2}. The second inequality uses an argument which is similar to the argument from the proof of Lemma \ref{lem.ac}:
\begin{align*}
\norm[\X]{ \frac{x-J_\lambda x}{\lambda} } & \leq \frac{1}{\lambda} \Var_{J_{\cdot} x} (\lambda ) = \frac{1}{\lambda} \int_0^\lambda \frac{\ud}{\ud\mu} \Var_{J_{\cdot} x} (\mu ) \; \ud\mu . 
\end{align*}

Let us turn to the inequality \ref{eq 4.1.4}. As mild solutions are also integral solutions, for every $x\in\overline{\dom{A}}$, every $(v,f) \in A$, and every $t\geq 0$, we have using   \eqref{eq 4.0.1} above,
\begin{equation}\label{eq 4.1.6} 
\norm[\X]{ v- S(t)x } - \norm[\X]{ v-x } \leq \int_0^t [v - S(s)x, f]\,\ud s + \omega \int_0^t \norm[\X]{ v - S(s)x}  \,\ud s ;
\end{equation}
Moreover, for every $\lambda > 0$, 
\begin{equation}\label{eq 4.1.7} 
[v - S(s)x, f] \leq \frac{1}{\lambda}\Big( \norm[\X]{ v - S(s)x + \lambda f }  - \norm[\X] {v - S(s)x} \Big). 
\end{equation} 
Choosing  $v=J_{\lambda}x$ and $f = A_\lambda x$ for some $\lambda >0$ with $\lambda \omega<1$ (here, $A_\lambda$ is again the Yosida approximation of $A$) and noting that $(v,f)\in A$, we have
\begin{align*}
[J_{\lambda}x - S(s)x , A_{\lambda}x ] & \leq \frac{1}{\lambda} \left (\norm[\X]{ J_{\lambda}x - S(s)x  + \lambda A_\lambda x}   - \norm[\X]{ J_{\lambda}x - S(s)x }\right) \nonumber \\
& = \frac{1}{\lambda} \left (\norm[\X]{ x - S(s)x}  - \norm[\X] { J_{\lambda}x - S(s)x }\right) . 
\end{align*}
Inserting this inequality into \eqref{eq 4.1.6} yields
\begin{align*}
\lefteqn{\norm[\X] { J_{\lambda}x - S(t)x}  - \norm[\X]{ x - J_\lambda x}} \\ 
& \leq \int_0^t\frac{1}{\lambda} \Big(\norm[\X]{ x - S(s)x}   - \norm[\X]{  J_{\lambda}x - S(s)x }\Big)\,\ud s + \omega \int_0^t \norm[\X]{ S(s)x - J_{\lambda}x}\,\ud s .
\end{align*}
Hence,
\begin{equation} \label{eq 4.1.8}
\begin{split}
\lefteqn{\norm[\X] { J_{\lambda}x - S(t)x}  - \norm[\X]{ x - J_\lambda x}} \\ 
& \leq \frac{1}{\lambda} \int_0^t \norm[\X]{x - S(s)x} \,\ud s - \frac{1-\lambda\omega}{\lambda} \int_0^t \norm[\X] { J_{\lambda}x - S(s)x }\,\ud s .
\end{split}
\end{equation}
Since 
\[
\norm[\X] {x - J_{\lambda}x} \leq \norm[\X]{ x - S(t)x } + \norm[\X]  {J_{\lambda}x - S(t)x} ,
\]
we have from \eqref{eq 4.1.8}  
\begin{align*}
\lefteqn{- \norm[\X]{ x - S(t)x } } \\
& \leq \frac{1}{\lambda} \int_0^t \norm[\X]{ x - S(s)x} \,\ud s  + \frac{1-\lambda\omega}{\lambda} \, \int_0^t \norm[\X]{x - S(s)x}    \,\ud s \\
& \phantom{\leq\ } - \frac{1-\lambda\omega}{\lambda} \int_0^t\norm[\X]{ x -J_{\lambda}x }\,\ud s \\ 
 & \leq  \frac{2 -\lambda \omega}{\lambda} \int_0^t \norm[\X]{ x - S(s)x } \,\ud s - \frac{t(1-\lambda\omega)}{\lambda} \norm[\X] {x - J_{\lambda}x } .
\end{align*}
Hence, 
\begin{equation}\label{eq 4.1.9} 
\norm[\X] {x - J_{\lambda}x} \leq \frac{\lambda}{t(1-\lambda\omega)} \norm[\X]{ x-S(t)x } + \frac{2-\lambda\omega}{t(1-\lambda\omega)}\int_0^t\norm[\X]{ x - S(s)x  } \,\ud s
\end{equation}  
Observe that for $t > 0$, 
\begin{align*}
\norm[\X]{ S(t)x - \frac{1}{t}\int_0^t S(s)x\,\ud s }  & \leq  \frac{1}{t} \int_0^t \norm[\X]{ S(t)x - S(s)x } \,\ud s\\
& \leq \frac{1}{t} \int_0^t e^{\omega s} \norm[\X] {S(t-s)x- x}   \,\ud s\\
& \leq \frac{e^{\omega t}}{t} \int_0^t  \norm[\X] {S(t-s)x- x}   \,\ud s\\
& = \frac{e^{\omega t}}{t} \int_0^t \norm[\X]{ S(s)x - x}   \,\ud s.
\end{align*}
Thus,
\begin{align*}
\norm[\X] {x - S(t)x } & \leq \norm[\X]{ x - \frac{1}{t}\int_0^t S(s)x\,\ud s } + \frac{e^{\omega t}}{t} \int_0^t \norm[\X]{ S(s)x - x  } \,\ud s\\
& \leq \frac{1+e^{\omega t}}{t} \int_0^t \norm[\X]{ S(s)x - x} \,\ud s ,
\end{align*}
and it follows from \eqref{eq 4.1.9} that 
\begin{align}
\lefteqn{\norm[\X]{ x - J_{\lambda}x }} \nonumber \\
& \leq \frac{\lambda}{1-\lambda\omega} \, \frac{1+ e^{\omega t}}{t^2} \int_0^t \norm[\X]{ S(s)x - x}  \,\ud s
+  \frac{2-\lambda\omega}{1-\lambda\omega} \, \frac{1}{t} \, \int_0^t \norm[\X] {S(s)x - x}  \,\ud s \nonumber  \\
& = \frac{\frac{\lambda(1+e^{\omega t})}{t}  +2-\lambda\omega}{1-\lambda\omega}\, \frac{1}{t} \, \int_0^t \norm[\X] {S(s)x - x}   \,\ud s. \label{thisone}
\end{align}
The last inequality in the statement follows from \eqref{thisone} by putting $\lambda = t$, assuming that $t\omega <1$. 
\end{proof}

From the pointwise estimates between the $K$-function and functions depending on the resolvent or the semigroup (Theorem \ref{th4.1}) it follows that the interpolation function $\cN_\BFS^\tau$ (and the interpolation function $\cN_\BFS$, if the underlying operator is surjective and $m$-accretive of type $0$) and some functions involving the resolvent of the operator $A$ and the semigroup generated by $-A$ are mutually equivalent. Here we say that two functions $f$, $g:\X\to [0,\infty ]$ are {\em equivalent} if there exist constants $c_1$, $c_2>0$ such that $c_1f\leq g \leq c_2f$. The constants $c_1$ and $c_2$ appearing in the following equivalences depend only on the type of the operator $A$, the interval $(0,\tau )$ and possibly a bound for the Hardy operator on the space $\BFS$, but not otherwise on the operator $A$.

In the following statement, there  appear functions depending on the resolvent of an operator, and these functions might not be defined on the entire interval $(0,\infty )$. However, we consider the products of such functions with the characteristic function $\chi_{(0,\tau )}$ interpreting the product as $0$ outside the interval $(0,\tau )$. 

\begin{corollary}\label{th4.2}
Let $A\subseteq \X \times \X$ be an $m$-accretive operator of type $\omega\geq 0$ on a Banach space $\X$, and let $S$ be the semigroup generated by $-A$. Let $\cN = (\cN_0 ,\cN_1 )$ be the associated interpolation couple, where $\cN_0 = \norm[\X]{\cdot}$ is the norm on $\X$ and $\cN_1 = |A\cdot |$ is the set norm of $A$. Let $\BFS$ be a Banach function space over the interval $(0,\infty )$, and assume that the Hardy operator is bounded on $\BFS$. Put $\BBFS = (\BFS ,\BFS )$. Then, for every $x\in \overline{\dom{A}}$ and for every $\tau >0$ with $\tau\omega <1$,
\begin{aufzi}
    \item \label{th4.2.a}
    \begin{equation*}
 	\frac12 \, \cN_\BFS^{\tau} (x) \leq \Norm[\BFS]{ \lambda\mapsto \frac{\norm[\X]{x - J_\lambda x}}{\lambda} \chi_{(0,\tau )} (\lambda )} \leq \frac{2-\tau\omega}{1-\tau\omega} \, \cN_\BFS^{\tau} (x) .
 	\end{equation*}
    \item \label{th4.2.b}
	\begin{multline*}
	(1-\tau\omega)\, \Norm[\BFS]{ \lambda \to \frac{\ud}{\ud\lambda} \Var_{J_\lambda x} \,\chi_{(0,\tau )} (\lambda ) } \leq \Norm[\BFS]{ \lambda\mapsto \frac{\norm[\X]{x - J_\lambda x}}{\lambda} \,\chi_{(0,\tau )} (\lambda ) } \\
    \leq \Norm[\mathcal{L}(\BFS)]{P} \, \Norm[\BFS]{ \lambda \to \frac{\ud}{\ud\lambda} \Var_{J_\lambda x} \,\chi_{(0,\tau )} (\lambda ) } . 
	\end{multline*}
    \item \label{th4.2.c}
    \begin{equation*}
 	\frac{1}{2\| P\|_{\cL (\BFS )}} \, \frac{1-\tau\omega}{3+e^{\omega\tau}}\, \cN_\BFS^\tau (x) \leq \Norm[\BFS]{ t \mapsto \frac{\norm[\X]{x - S(t) x}}{t} \chi_{(0,\tau )} (t) } \leq (1+e^{\omega\tau}) \, \cN_\BFS^\tau (x) . 
 	\end{equation*}
	\item \label{th4.2.d}
	\begin{equation*}
	\frac12 \,\cN_\BBFS^{C,\tau} (x) \leq \Norm[\BFS]{ \lambda\mapsto \frac{\norm[\X]{x - J_\lambda x}}{\lambda} \, \chi_{(0,\tau )} (\lambda )} \leq \frac{2-\tau\omega}{1-\tau\omega} \, \cN_\BBFS^{L^0,\tau} (x) 
	\end{equation*}
	\item \label{th4.2.e} If $\X$ has the Radon-Nikodym property, then 
	\begin{equation*}
	\cN_\BBFS^{tr,\tau} (x) \leq \frac{2}{1-\tau\omega} \, \Norm[\BFS]{ \lambda\mapsto \frac{\norm[\X]{x - J_\lambda x}}{\lambda} \chi_{(0,\tau )} (\lambda )}
	\end{equation*}
	\end{aufzi}
	If the operator $A$ is $m$-accretive of type $0$ and if $A$ is in addition surjective, then the inequalities in \ref{th4.2.a}--\ref{th4.2.d} hold for the interpolation functions $\cN_\BFS$ (instead of $\cN_\BFS^\tau$) and $\cN_\BBFS^C$ and $\cN_\BBFS^{L^0}$ (instead of $\cN_\BBFS^{C,\tau}$ and $\cN_\BBFS^{L^0,\tau}$) and without the characteristic functions $\chi_{(0,\tau )}$ in the $\BFS$-norms.
	\end{corollary}
	
	\begin{proof}
The statements are direct consequences of the pointwise estimates proved in Theorem \ref{th4.1}, the ideal property of the space $\BFS$ and, if necessary, the boundedness of the Hardy operator on $\BFS$. Statement \ref{th4.2.a} 
follows directly from Theorem \ref{th4.1} \ref{eq 4.1.1}, and statement \ref{th4.2.b} 
follows from Theorem \ref{th4.1} \ref{eq 4.1.2}. For the statement \ref{th4.2.c} 
one combines Theorem \ref{th4.1} \ref{eq 4.1.1}, \ref{eq 4.1.2} and \ref{eq 4.1.4}. 

From the definition of $\cN_{\BBFS}^{C,\tau}$ as an infimum we have
    \begin{align*}
    \cN_{\BBFS}^{C,\tau}(x) & \leq \Norm[\BFS]{ \lambda \to \frac{\norm[\X]{x - J_{\lambda}x}}{\lambda} \, \chi_{(0,\tau )} (\lambda )} + \Norm[\BFS]{\lambda \to \vert AJ_{\lambda}x \vert\, \chi_{(0,\tau )} (\lambda )}\\
    & \leq 2  \Norm[\BFS]{ \lambda \to \frac{\norm[\X]{x - J_{\lambda}x}}{\lambda} \, \chi_{(0,\tau )} (\lambda )} ,
\end{align*}	
which is the first inequality in \ref{th4.2.d}; note also that the resolvent is continuous up to $0$ and $\lim_{\lambda\to 0+} J_\lambda x = x$ for every $x\in\overline{\dom{A}}$. The second inequality in \ref{th4.2.d} follows from statement \ref{th4.2.a} above, combined with Lemma \ref{lem.k-mean} ($\cN_\BFS^\tau \leq \cN_\BBFS^{L^0,\tau}$).

Finally, assume that $\X$ has the Radon-Nikodym property. By \eqref{eq.res.lip}, for every $x\in\overline{\dom{A}}$ the function $t\mapsto J_{t} x$ is locally Lipschitz continuous on $(0,\tau )$ for every $\tau >0$ with $\omega\tau <1$. The Radon-Nikodym property implies that this function actually is in $W^{1,\infty}_{loc} ((0,\tau );\X )$, and in particular it is differentiable almost everywhere. By the equality \eqref{eq.dotf.varf} and by Theorem \ref{th4.1} \ref{th4.2.b}, for every $x\in\overline{\dom{A}}$ and for almost every $\lambda\in (0,\tau )$, 
\begin{align*}
\norm[\X]{\frac{\ud}{\ud\lambda} J_\lambda x} & = \frac{\ud}{\ud\lambda} \Var_{J_\cdot x} (\lambda ) \\
& \leq \frac{1}{1-\lambda\omega} \frac{\norm[\X]{x - J_{\lambda}x}}{\lambda} .
\end{align*}
Now, if the function $\lambda\mapsto \frac{\norm[\X]{x-J_\lambda x}}{\lambda} \chi_{(0,\tau )} (\lambda)$ belongs to $\BFS$, then, by the preceding inequality, the function $w (\lambda ) := J_\lambda x$ belongs to $W^{1,1} (0,\tau ;\X )$, $w(0) = x$, $\dot{w} \, \chi_{(0,\tau )}\in\BFS$ and, therefore,
\begin{align*}
    \cN_{\BBFS}^{tr,\tau} (x) & \leq \Norm[\BFS]{ \lambda \to \norm[\X]{\frac{\ud}{\ud\lambda} J_\lambda x} \, \chi_{(0,\tau )} (\lambda )} + \Norm[\BFS]{t \to \vert AJ_{\lambda}x \vert\, \chi_{(0,\tau )} (\lambda )}\\
    & \leq \frac{2}{1-\tau\omega}  \Norm[\BFS]{ \lambda \to \frac{\norm[\X]{x - J_{\lambda}x}}{\lambda} \, \chi_{(0,\tau )} (\lambda )} .
\end{align*}
\end{proof}

In the following, besides the operator $A$, we consider also the operators $I+hA$ for $h>0$. If $A$ is $m$-accretive of type $\omega\in\R$ and if $h\omega <1$, then $I+hA$ is $m$-accretive of type $0$ {\em and} surjective.   Accordingly, and with some slight abuse of notation, we consider the interpolation couples $\cN^A = (\norm[\X]{\cdot} , |A\cdot |)$ and $\cN^{I+hA} = (\norm[\X]{\cdot} , |(I+hA)(\cdot )| )$ and the associated $K$-functions $K^A$ and $K^{I+hA}$. 

\begin{corollary} \label{th4.3}
Let $A$ be an $m$-accretive operator of type $\omega\geq 0$ on a Banach space $\X$, and let $S^A$ be the semigroup generated by $ -A$ and $S^{I+hA}$ be the semigroup generated by $-(I+hA)$ ($h>0$). Let $\cN^A = (\norm[\X]{\cdot} , |A\cdot |)$ and $\cN^{I+hA} = (\norm[\X]{\cdot} , |(I+hA)(\cdot )| )$ be the associated interpolation couples, and $K^A$ and $K^{I+hA}$ the associated $K$-functions. Let $\BFS$ be a Banach function space over the interval $(0,\infty)$, and assume that the Hardy operator is bounded on $\BFS$. Let $\BBFS = (\BFS , \BFS )$. Then for any $\tau$, $h >0$ with  $\tau \omega<1$ and $h\omega <1$, 
\begin{align*} 
\dom{A} & \subseteq  
  \left\{ x \in \X \st t \mapsto \frac{ \norm[\X]{ x - J^{A}_t x }}{t}\chi_{(0,\tau)}(t)  \in \BFS \right\} \\
&=   \left\{ x \in \X \st t \mapsto \frac{ \norm[\X]{ x - J^{I+hA}_t x }}{t}  \in \BFS \right\} \\
&= \left\{ x \in \overline{ \Dom{A}} \st  t\mapsto \frac{\norm[\X]{ x - S^{A}(t)x}}{t}\chi_{(0,\tau)(t)} \in \BFS   \right\} \\
&= \left\{ x \in \overline{ \Dom{A}} \st  t\mapsto \frac{\norm[\X]{ x - S^{I+hA}(t)x}}{t} \chi_{(0,\tau )} (t) \in \BFS   \right\} \\
& = \left\{ x \in \X \st t \mapsto \frac{ K^{A}(x,t) }{t}\chi_{(0,\tau)}(t) \in \BFS \right\}  = \dom{(\cN^A)_\BFS^{\tau}} \\
& = \left\{ x \in \X \st t \mapsto \frac{ K^{I+hA}(x,t) }{t}\chi_{(0,\tau)}(t) \in \BFS \right\}  = \dom{(\cN^{I+hA})_\BFS^{\tau}} \\
& = \left\{ x \in \X \st t \mapsto \frac{ K^{I+hA}(x,t) }{t} \in \BFS \right\}  = \dom{(\cN^{I+hA})_\BFS} \\
& = \left\{ x \in \X \st t \mapsto \frac{\ud}{\ud t} \Var_{{J^{A}_t}u}\,\chi_{(0,\tau)}(t) \in \BFS \right\}  \\
& = \left\{ x \in \X \st t \mapsto \frac{\ud}{\ud t} \Var_{{J^{I+hA}_t}u} \in \BFS \right\}  \\
& = \dom{(\cN^A)_\BBFS^{L^0,\tau}} = \dom{(\cN^A)_\BBFS^{C,\tau}} \\
& = \dom{(\cN^{I+hA})_\BBFS^{L^0,\tau}} = \dom{(\cN^{I+hA})_\BBFS^{C,\tau}} \\
& \subseteq \overline{\dom{A}}.
\end{align*}
If the Banach space $\X$ has the Radon-Nikodym property, then one has in addition the equality 
\[
\dom{(\cN^A)_\BFS^\tau} = \dom{(\cN^A)_\BFS^{tr,\tau}} .
\]
\end{corollary}

\begin{proof}
We first note here that the boundedness of the Hardy operator implies that for any $ x \in X, $ $ t \mapsto \frac{ \norm[\X]{ x - J^{A}_t x }}{t} \chi_{(0,\tau)}(t)  \in \BFS$ implies that $ x \in \overline{\dom{A}}.$
Now, the equalities of the sets involving the operator $A$, that is, the sets appearing in the first, third, fifth, eight and tenth line are a consequence of Corollary \ref{th4.2}; recall in particular the obvious inequality $(\cN^A)_\BBFS^{L^0,\tau} \leq (\cN^A)_\BBFS^{C,\tau} $. 
Similarly, the equalities of the sets involving the operator $I+hA$ and appearing in the second, fourth, sixth, ninth and eleventh line are a consequence of Corollary \ref{th4.2}. The equality $\dom{(\cN^{I+hA})_\BFS^{\tau}} = \dom{(\cN^{{I+hA}})_\BFS}$ follows from Lemma \ref{3.11}; in order to be able to apply this lemma we use the fact that $I+hA$ is surjective so that the set norm $|(I+hA)(v_0)|$ is equal to $0$ for some $v_0\in \X$. The remaining equality $\dom{(\cN^A)_\BFS^{\tau}} = \dom{(\cN^{I+hA})_\BFS^{\tau}}$ uses the following estimates. For every $x$, $v \in \X$ and $t\in (0,\tau )$,
\begin{align*}
\frac{ K^{I+hA}(x,t) }{t} &\leq \frac{\Norm[\X]{x-v}}{t} + |(I+hA)(v)| \\
& \leq \frac{\Norm[\X]{x-v}}{t} + h|Av| + \frac{\tau}{t}\Norm[\X]{x-v} +  \Norm[\X]{x} .
\end{align*}
Taking the infimum over all $v\in \X$ yields
\begin{equation*}
\frac{ K^{I+hA}(x,t) }{t} \leq \max\{ 1+\tau,h\} \, \frac{ K^{A}(x,t) }{t} + \Norm[\X]{x} 
\end{equation*}
for $t\in (0,\tau )$. Conversely for every $x$, $v \in \X$ and $t\in (0,\tau )$,
\begin{align*}
\frac{ K^{A}(x,t) }{t} 
& \leq \frac{\Norm[\X]{x-v}}{t} + |Av| \\
& \leq \frac{\Norm[\X]{x-v}}{t} + \frac{|(I+hA)(v)|}{h} + \frac{\tau}{h}\frac{\Norm[\X]{x-v}}{t} + \Norm[\X]{x} .
\end{align*}
Taking the infimum over all $v\in \X$ yields
\begin{equation*}
\frac{ K^{A}(x,t) }{t} \leq  \max \{ \frac{1}{h}, 1+\frac{\tau}{h} \} \, \frac{ K^{(I + hA)}(x,t) }{t} + \Norm[\X]{x}.
\end{equation*}
Multiplying these inequalities with the characteristic function $\chi_{(0,\tau )}$ and then taking the norm in $\BFS$ yields the remaining equality.

Concerning the two inclusions, note that for $x\in\dom{A}$ and $t>0$, $\frac{K(x,t)}{t} \leq |Ax|$, that is, the function $t\mapsto \frac{K(x,t)}{t}$ is bounded on every bounded interval $(0,\tau )$. Since $L^1 \cap L^\infty (0,\infty )\subseteq\BFS$, this implies $t\mapsto \frac{K(x,t)}{t} \chi_{(0,\tau )} (t) \in\BFS$. Hence, $(\cN^A)_\BFS^\tau (x) <\infty$, that is, $x\in\dom{(\cN^A)_\BFS^\tau}$, and we have proved the first inclusion.

The second inclusion follows similarly. For $x\not\in\overline{\dom{A}}$ and $t>0$, $\frac{K(x,t)}{t} \geq \frac{C}{t}$, where $C = {\rm dist}\, (x,\dom{A}) >0$. In particular, the function $t\mapsto \frac{K(x,t)}{t}$ is not integrable on any bounded interval $(0,\tau )$, and hence not an element of $\BFS$. It follows that $x\not\in\dom{(\cN^A)_\BFS^\tau}$.

Finally, assume that $\X$ has the Radon-Nikodym property. Then, by Corollary \ref{th4.2} \ref{th4.2.d} and \ref{th4.2.e}, $(\cN^A)_\BBFS^{tr,\tau} \leq \frac{2(2-\tau\omega)}{(1-\tau\omega)^2} (\cN^A)_\BBFS^{L^0,\tau}$. On the other hand, by a trivial estimate and by Lemma \ref{lem.trace.mean}, $(\cN^A)_\BBFS^{L^0,\tau} \leq (\cN^A)_\BBFS^{W^{1,1},\tau} \leq \Norm[\cL (\BFS )]{P} \, (\cN^A)_\BBFS^{tr,\tau}$, and hence, the interpolation functions $(\cN^A)_\BBFS^{tr,\tau}$ and $(\cN^A)_\BBFS^{L^0,\tau}$ are equivalent. Together with the above equalities of domains this yields the final claim.
\end{proof}

The following theorem is the motivation for the interest in interpolation functions and their effective domains in the context of interpolation associated with $m$-accretive operators. The effective domains of interpolation functions are somehow interpolation sets between the domain of an operator $A$ and the closure of its domain (have in mind, however, Remark \ref{rem.k} (b)). Recall that the semigroup orbits corresponding to initial values in the domain of the generator are Lipschitz continuous, while those corresponding to initial values in the closure of the domain are merely continuous. If an initial value belongs to the effective domain of an interpolation function, then the corresponding semigroup orbit has a prescribed regularity between Lipschitz continuity and mere continuity. The interpolation space between the {\em Banach spaces} $C$ and $Lip$ appearing in the following theorem is the classical real interpolation space, with the exception perhaps that we do not restrict ourselves to interpolation using polynomially weighted $L^p$-spaces but that we work with more general Banach function spaces.

\begin{theorem}
Let $A$ be an $m$-accretive operator of type $\omega\geq 0$ on a Banach space $\X$, and let $S$ be the semigroup generated by $ -A$. Let $\cN^A = (\norm[\X]{\cdot} , \vert A\cdot \vert )$ be the interpolation couple associated with $A$. Let $\BFS$ be a Banach function space over the interval $(0,\infty)$, and assume that the Hardy operator is bounded on $\BFS$. Then, for every $x\in\dom{(\cN^A)_\BFS^\tau}$ and for every $T\geq 0$,
\[
 S(\cdot )x \in (C([0,T];\X ) , \textrm{Lip} ([0,T];\X ))_\BFS .
\]
\end{theorem}

\begin{proof}
Let $K(\cdot , \cdot , C([0,T];\X ) , \textrm{Lip} ([0,T];\X ) )$ be the $K$-function associated with the interpolation between the supremum norm and the Lipschitz norm, and let $K^A$ be the $K$-function associated with the interpolation couple $\cN^A$. Recall that for every $v\in\dom{A}$ the function $S(\cdot )v$ is Lipschitz continuous on the interval $[0,T]$ with Lipschitz constant $e^{\omega T} \, |Av|$ (see \eqref{eq 4.0.2 }). Hence, for every $t>0$,
 \begin{align*}
  \lefteqn{K(S(\cdot )x , t , C([0,T];\X ) , \textrm{Lip} ([0,T];\X ) )} \\
  & = \inf\{\Norm[\infty]{S(\cdot)x - g}+ t\Norm[Lip]{g} \st g\in \textrm{Lip}([0,T];\X)\}\\
  &\leq \inf\{ \Norm[\infty]{S(\cdot)x - S(\cdot)v} + t \Norm[Lip]{S(\cdot)v} \st v \in \dom{A}\}\\
  &\leq e^{\omega T}\inf \{ \Norm[\X]{x-v} + t  |Av| \st v\in \dom{A}\}\\
  &= e^{\omega T} K^A (x,t)
 \end{align*}
 and the result follows.
\end{proof}

\begin{theorem}[H\"older type estimate] \label{cor.holder}
Let $A$ be an $m$-accretive operator of type $\omega\geq 0$ on a Banach space $\X$, and let $S$ be the semigroup generated by $ -A$. Let $\cN^A = (\norm[\X]{\cdot} , \vert A\cdot \vert )$ be the interpolation couple associated with $A$. Let $\BFS$ be a Banach function space over the interval $(0,\infty)$, and assume that the Hardy operator is bounded on $\BFS$. Then, for every $\tau >0$ with $\tau \omega<1$, $x\in\dom{(\cN^A)_\BFS^\tau}$,  $T> 0$ and $t\in[0,T] $, $h\in (0,\tau)$ one has,
\[
 \norm[\X]{S(t+h)x - S(t)x} \leq e^{\omega T}(1+ e^{\omega \tau})\, \cN_\BFS^{\tau} (x)\, \frac{h}{\Norm[\BFS]{\chi_{(0,h)}}} .
\]

\end{theorem}
\begin{proof}
 For any $x\in \X$, the map $t\in (0,\infty) \to \frac{K(x,t)}{t}$ is decreasing. Thus, for every $s$, $h\in (0,\infty)$ with $s\leq h$ we have 
 \begin{equation*}{\label{eq4.5.1}}
\frac{K(x,h)}{h} \leq \frac{K(x,s)}{s}.
\end{equation*}
This inequality and the Banach ideal property of $\BFS$ imply that for every $x\in\X$ and every $h\in (0,\tau )$ 
\begin{align*}
 \frac{K(x,h)}{h} \Norm[\BFS]{\chi_{(0,h)}} & \leq \Norm[\BFS]{s \mapsto \frac{K(x,s)}{s} \chi_{(0,h)} (s) } \\
 & \leq \Norm[\BFS]{s \to \frac{K(x,s)}{s} \chi_{(0,\tau)}(s) }\\
 &= \cN_{\BFS}^{\tau}(x) . 
\end{align*}   
For every $x\in\dom{\cN_\BFS^{\tau}}$, the right hand side is finite. From  Theorem \ref{th4.1} \ref{eq 4.1.3} and the preceding inequality we have for every $t\in [0,T ]$, $h\in (0,\tau)$
\begin{align*}
\norm[\X]{S(t+h)x - S(t)x}  & = \Norm[\X]{S(t)S(h)x - S(t)x}\\
& \leq e^{\omega t } \Norm[\X]{S(h)x - x}\\
& \leq  e^{\omega t}(1+ e^{\omega h }) K(x,h)\\
& \leq  e^{\omega T}(1 + e^{\omega \tau})\, \cN_\BFS^{\tau} (x) \, \frac{h}{\Norm[\BFS]{\chi_{(0,h)}}}.   
\end{align*}

\end{proof}

\begin{remark}
If, in the previous theorem, the underlying Banach function space $\BFS$ is the polynomially weighted $L^p$ space
\[
\BFS_{\theta , p} = L^p (0,\infty ; t^{p(1-\theta) -1}\,\ud t) ,
\]
used in classical interpolation theory, then, for some constant $C\geq 0$ and every $h>0$,
\[
\Norm[\BFS_{\theta ,p}]{\chi_{(0,h)}} = C \, h^{1-\theta} .
\]
Therefore, if $x\in\dom{(\cN^A)_{\BFS_{\theta , p}}^\tau}$, then the semigroup orbit $S(\cdot )x$ is locally H\"older continuous of H\"older exponent $\theta$. 
\end{remark}

Let $A$ and $B$ be two operators on a Banach space $\X$.  We say $B$ is {\em dominated} by $A$ if there exist $a\in (0,1)$ and a continuous function $b: \mathbb{R}_{+} \to \mathbb{R}_{+}$ such that 
 \[ 
 |Bx| \leq a \, |Ax|+ b(\Norm[\X]{x}) \text{ for every } x\in \X.
 \]
 It follows from this inequality that if $B$ is dominated by $A$, then $\dom A \subseteq \dom B$.
 
 In the following theorem, we compare the interpolation function resulting from the interpolation between the norm on $\X$ and the set norm of the operator $A$, and the interpolation function resulting from the interpolation between the norm on $\X$ and the set norm of $A+B$ for an appropriate perturbation $B$. 
 
\begin{theorem}[Perturbation theorem] \label{thm.perturbation}
 Let $A$ be an $m$-accretive operator of type $\omega_A \geq 0$ on a Banach space $\X$. Let $B$ be a second operator on $\X$ such that $B$ is dominated by $A$ and that $A+B$ is $m$-accretive of some type $\omega_{A+B} \geq 0$. Let $\cN^A = (\norm[\X]{\cdot} , |A\cdot |)$ and $\cN^{A+B} = (\norm[\X]{\cdot} , |(A+B)\cdot |)$ be the interpolation couples associated with $A$ and $A+B$, respectively. Let $\BFS$ be a Banach function space over the interval $(0,\infty )$, and suppose that the Hardy operator is bounded on $\BFS$. Then, for every $\tau >0$ with $\tau \, \max \{ \omega_A , \omega_{A+B} \} <1$, there exists a constant $\tilde{a} \geq 0$ and an increasing function $\tilde{b} :\R_+ \to\R_+$ such that, for every $x\in\X$,
 \begin{equation} \label{4.6.0}
     (\cN^{A+B})^{\tau}_\BFS (x) \leq \tilde{a} \, (\cN^{A})^{\tau}_\BFS (x) + \tilde{b} (\norm[\X]{x} ) ,
 \end{equation}
 and in particular,
 \begin{equation}\label{4.6.1}
\dom (\cN^A)_{\BFS}^{\tau} \subseteq \dom (\cN^{A+B})_{\BFS}^{\tau} .
 \end{equation}
 Moreover, if $B$ is single valued, then
 \begin{equation}\label{4.6.2}
 \dom (\cN^A)_{\BFS}^{\tau} = \dom (\cN^{A+B})_{\BFS}^{\tau} .
 \end{equation}
\end{theorem}

The proof shows that the constant $\tilde{a}$ and the function $\tilde{b}$ in the above statement depend on the constant $a$ and the function $b$ from the domination condition, the type $\omega_A$ of the operator $A$, the constant $\tau$, the norm $\Norm[\BFS]{\chi_{(0,\tau )}}$ of the characteristic function in the space $\BFS$, and the supremum norm of $t\mapsto \norm[\X]{J_t^A x_0}$ over the interval $(0,\tau )$ for some arbitrary, but fixed element $x_0\in\X$. 

\begin{proof}
By assumption, there exist a constant $a\in (0,1)$ and an increasing function $b:\mathbb{R}_{+} \to \mathbb{R}_{+}$ such that 
\[ 
|Bv| \leq a\, |Av| + b(\Norm[\X]{v}) \text{ for every } v\in \dom A. 
\]
Notice that for the first part of the proof, we only need the constant $a$ to be positive.

Let $x\in\X$. We denote by $K^{A+B} (x,\cdot )$ the $K$-function associated to the interpolation couple $\cN^{A+B}$. For every $t>0$ and every $v\in\dom{(A+B)} = \dom{A}$ we have
\begin{equation} \label{eq.kab}
K^{A+B} (x,t) \leq \norm[\X]{x-v} + t \, |(A+B) v| .
\end{equation}
For every $v\in \dom{(A+B)} = \dom{A}$, by the domination assumption,
\begin{align*}
|(A+B)v| & = \inf\{\Norm[\X]{f_1 +f_2} \st f_1\in Av, f_2 \in Bv \} \\
& \leq \inf\{ \Norm[\X]{f_1} \st f_1\in Av\} + \inf\{ \Norm[\X]{f_2}\st f_2 \in Bv\} \\
& = |Av| + |Bv| \\
& \leq (a+1)\, |Av| + b (\norm[\X]{v}) .
\end{align*}
We use this inequality in \eqref{eq.kab} with the special element $v=J_t^A x$ (where $J_t^A$ is the resolvent of $A$!). It follows that for every $t\in (0,\tau )$
\begin{equation} \label{eq.kab2}
\begin{split}
\frac{K^{A+B} (x,t)}{t} & \leq \frac{\norm[\X]{x-J_t^Ax}}{t} + (a+1) \, |AJ_t^A x| + b(\norm[\X]{J_t^A x} ) \\
& \leq (a+2) \, \frac{\norm[\X]{x-J_t^A x}}{t} + b(\norm[\X]{J_t^A x} ) .
\end{split}
\end{equation}
Fix now an arbitrary element $x_0\in\X$, independently of $x$. Then, for every $t\in (0,\tau )$,
\begin{align*}
\norm[\X]{J_t^A x} & \leq \norm[\X]{J_t^A x - J_t^A x_0} + \norm[\X]{J_t^A x_0} \\
& \leq \frac{1}{1-t\omega} \norm[\X]{x-x_0} + \norm[\X]{J_t^A x_0} \\
& \leq \frac{1}{1-\tau\omega} \norm[\X]{x} + C ,
\end{align*}
where $C\geq 0$ is a constant depending only on $x_0$, $\tau\omega$ and the supremum of $\norm[\X]{J_t^A x_0}$ over $t\in (0,\tau )$, but which is independent of $x$. Seeing the inequality \eqref{eq.kab2} as an inequality of functions of $t\in (0,\tau )$, extending all functions by $0$ outside this interval, and taking the $\BFS$-norm on both sides yields
\[
(\cN^{A+B})^{\tau}_\BFS (x) \leq \Norm[\BFS]{t\mapsto \frac{\norm[\X]{x-J_t^A x}}{t} \, \chi_{(0,\tau )} (t) } + b( \frac{1}{1-\tau\omega} \norm[\X]{x} + C ) \, \Norm[\BFS]{\mathbf{\chi_{(0,\tau )}}} ,
\]
where $(\cN^{A+B})^{\tau}_\BFS$ is the "finite" interpolation function of $\cN^{A+B}$ with respect to the $K$-method. By Corollary \ref{th4.2}, the resulting first term on the right hand side is comparable to (dominated by) $(\cN^A)_\BFS^{\tau} (x)$ with a constant independent of $x$. Define the second term on the right hand side to be $\tilde{b} (\norm[\X]{x} )$, and the proof of the inequality \eqref{4.6.0} is complete. The inclusion \eqref{4.6.1} is a direct consequence. 

If $B$ is single valued, then for every $v\in\X$
\begin{align*}
|Av| &= |Av + Bv -Bv|\\
& \leq |Av + Bv| + |Bv|\\
& \leq |Av + Bv| + a\, |Av| +b (\Norm[\X]{v}) .
\end{align*}
Therefore, for every $v\in\X$,
\begin{align*}
|-Bv| & \leq a\, |Av| + b(\norm[\X]{v}) \\
 & \leq \frac{a}{1-a}  |(A+B)v| + \frac{2}{1-a} \, b(\Norm[\X]{v}) ,
\end{align*}
that is, $-B$ is dominated by $A+B$, with the exception that the constant $\frac{a}{1-a}$ appearing in the preceding inequality is now greater than $1$. However, in the first part of the proof, we have not used the assumption that this constant belongs to $(0,1)$. It follows that $-B$ is (up to the change of this small assumption on the constant) dominated by $A+B$. By repeating the above proof with $A$ replaced by $A+B$, and with $B$ replaced by $-B$, we find an inequality of the type
\[
(\cN^{A})^{\tau}_\BFS (x) \leq a' \, (\cN^{A+B})^{\tau}_\BFS (x) + b' (\norm[\X]{x} ) \quad (x\in\X ),
\]
and this yields the converse inclusion 
\[
\dom{(\cN^{A+B})^{\tau}_\BFS} \subseteq \dom{(\cN^{A})^{\tau}_\BFS} 
\]
needed for the proof of \eqref{4.6.2}.
\end{proof}

\section{Regularizing semigroups}

Let $A \subseteq X \times X$ be an $m$-accretive operator of type $\omega\in\R$ on a Banach space $\X$, and let $S$ be the semigroup generated by $-A$. This semigroup is said to be a {\em regularizing semigroup} if 
\begin{aufzi}
    \item for each $t>0$, $S(t)$ maps $\overline{\dom A}$ into $\dom{A}$, and
    \item there exist $C\geq 0$ and $\tau >0$ such that, for each $t\in (0,\tau)$ and $x\in \overline{\dom A}$,
    \begin{equation*}
        \vert {AS(t)x} \vert \leq C\,\frac{K(x,t)}{t}.
    \end{equation*}
\end{aufzi}

\begin{example}
(a) If $S$ is a {\em linear} $C_0$-semigroup, then $S$ is regularizing if and only if it is an analytic semigroup. This has been remarked in  \cite{Bre76} and \cite{Df81}, but without proof. The proof of this remark might be well known, but we are unable to find a reference, and for the reader's convenience, we give here a short proof.

\begin{proof}[Proof of this statement]
By \cite[Theorem 5.2, Chapter 2]{Pa83}, a $C_0$-semigroup $S$ with generator $-A$ is analytic if and only if, for some / every $\tau >0$ and for some finite $C\geq 0$ (depending possibly on $\tau$),
    \begin{equation} {\label{eq5.1}}
        \sup_{t\in (0,\tau)} \norm[\mathcal{L}(X)]{tAS(t)} \leq C .
    \end{equation}

Assume first that $S$ is a regularizing $C_0$-semigroup. Then there exists $C\geq 0$ and $\tau >0$ such that, for each $t\in (0,\tau)$ and $x\in \X$,
    \begin{align*}
    \norm[\X]{tAS(t)x}  & \leq C\, K(x,t) \leq C\, \norm[\X]{x} ;
    \end{align*}
    The second inequality follows from \eqref{lineark} by putting $x_0 = x $ and $x_1 = 0$. 
Hence $S$ is analytic.

For the converse, assume that $S$ is an analytic semigroup, that is, \eqref{eq5.1} holds for some / every $\tau >0$ and for some $C\geq 0$. Then, for every $x\in\X$, $v\in \dom{A}$ and for every $t\in (0,\tau )$,
    \begin{align*}
        \norm[\X]{tAS(t)x} & \leq \norm[\X]{tAS(t)(x-v)} + \norm[\X]{tAS(t)v} \\
        & \leq C \,\norm[\X]{x-v} + e^{\omega\tau} \, t\norm[\X]{Av}\\
        & \leq \max \{ C, e^{\omega\tau} \} \, ( \norm[\X]{x-v} + t\norm[\X]{Av})
    \end{align*}
    Taking the infimum on the right hand side over all $v\in \dom{A}$ and recalling the definition of the $K$-functional yields that the semigroup $S$ is regularizing.
\end{proof}

(b) Every semigroup generated by the negative subgradient of a lower semicontinuous, semiconvex, proper function $\mathcal{E}$ on a Hilbert space is regularizing by \cite[Th\'eor\`eme 3.2]{Br73}.
\end{example}

For regularizing semigroups on Banach spaces we have the following additional equivalence between interpolation functions and the resulting characterisation of their domains.

\begin{theorem}{\label{th5.3}}
 Let $A \subseteq X \times X$ be an $m$-accretive operator of type $\omega\in\R$ on a Banach space $\X$, and let $S$ be the semigroup generated by $-A$. Assume that this semigroup is regularizing. Let $\cN^A = (\norm[\X]{\cdot} , \vert A\cdot \vert )$ be the interpolation couple associated with $A$. Let $\BFS$ be a Banach function space such that the Hardy operator is bounded on $\BFS$. Then, for every $\tau >0$ and for every $x\in\overline{\dom{A}}$,  
\begin{equation} {\label{eq5.3}}
 \frac{1}{e^{\omega\tau}\, \norm[\mathcal{L}(\BFS)]{P} +1}\, (\cN^A)_{\BFS}^{\tau}(x) \leq \norm[\BFS]{t \mapsto \vert AS(t)x \vert\,\chi_{(0,\tau)}(t)} \leq     (\cN^A)_{\BFS}^{\tau}(x) ,
\end{equation}
 and in particular,
 \begin{equation*}
     \dom{(\cN^A)_{\BFS}^{\tau}} = \{x\in \overline{\dom{A}} \st t \mapsto \vert AS(t)x \vert\,\chi_{(0,\tau)}(t) \in \BFS \}
 \end{equation*}
\end{theorem}

\begin{proof}
    The second inequality of \eqref{eq5.3} follows from the definition of a regularizing semigroup, from the definition of the interpolation function $(\cN^A)_{\BFS}^{\tau}$ and from the ideal property of $\BFS$.
    
    In order to show the first inequality in \eqref{eq5.3}, let $x\in\overline{\dom{A}}$. By definition of the $K$-function and since $S(t)x\in\dom{A}$ for $t>0$,
    \begin{equation} {\label{eq5.5}}
        \frac{K(x,t)}{t} \leq \frac{\norm[\X]{x-S(t)x}}{t} + \vert AS(t)x \vert. 
    \end{equation}
    Moreover, since $S(s)x\in\dom{A}$ for every $s>0$, the function $[0,\tau ] \to \X$, $t \mapsto S(t)S(s)x$ is Lipschitz continuous with Lipschitz constant $e^{\omega\tau}\, \vert AS(s)x \vert$. In other words, for every $s$, $t_1$, $t_2 \in (0,\tau )$ with $t_1 \neq t_2$ we have
    \begin{equation*}
         \frac{\norm[\X]{S(t_1+s)x - S(t_2+s)x}}{|t_1-t_2|} \leq e^{\omega\tau}\,\vert AS(s)x \vert ,
    \end{equation*}
    which implies for all $t\in (0,\tau)$ that
    \begin{equation*}
        \frac{\ud}{\ud t}\Var_{S(\cdot)x}(t) \leq e^{\omega\tau}\, \vert AS(t)x \vert .
    \end{equation*}
    Hence, for every $t\in (0,\tau )$, 
    \begin{equation*}
        \frac{\norm[\X]{x- S(t)x}}{t} \leq \frac{1}{t}\int_0^t \frac{\ud}{\ud s}\Var_{S(\cdot)x}(s)\,\ud s \leq e^{\omega\tau}\, \frac{1}{t} \int_0^{t} \vert AS(s)x \vert \,\ud s .
    \end{equation*}
    from where it folllows that
    \begin{align*} {\label{eq5.4}}
 \norm[\BFS]{t \mapsto \frac{\norm[\X]{x-S(t)x}}{ t}\,\chi_{(0,\tau)}(t)} & \leq e^{\omega\tau}\, \norm[\BFS]{ t \mapsto \frac{1}{t} \int_0^{t} \vert AS(s)x \vert \,\ud s \,\chi_{(0,\tau)}(t) } \nonumber\\
 &\leq e^{\omega\tau}\, \norm[\mathcal{L}(\BFS)]{P}  \norm[\BFS]{t \mapsto \vert AS(t)x \vert\,\chi_{(0,\tau)}(t)}
    \end{align*}
  Combining this inequality with \eqref{eq5.5} yields
    \begin{equation*}
        (\cN^A)_{\BFS}^{\tau}(x) \leq    (e^{\omega\tau}\, \norm[\mathcal{L}(\BFS)]{P} +1)  \norm[\BFS]{t \mapsto \vert AS(t)x \vert\,\chi_{(0,\tau)}(t)} ,
    \end{equation*}
 which gives the first inequality in \eqref{eq5.3}. The equality of the effective domains is a direct consequence of \eqref{eq5.3}. 
\end{proof}

\section{Interpolation associated with subgradients}

Let $H$ be a Hilbert space with inner product $\langle\cdot ,\cdot\rangle_H$, and let $\mathcal{E}:H\to [0,\infty]$ be a proper, lower semicontinuous and convex function. Then the {\em subgradient}
\begin{equation}
\partial\mathcal{E} := \{ (x,f)\in H\times H \st x\in\dom{\mathcal{E}} \text{ and for all } v\in H : \mathcal{E} (v) - \mathcal{E} (x) \geq \langle f,v-x\rangle_H \}    
\end{equation}
is an $m$-accretive operator on $H$ \cite[Exemple 2.3.4, p. 25]{Bre73}, and $\overline{\dom{\partial\mathcal{E}}} = \overline{\dom{\mathcal{E}}}$ \cite[Proposition 2.11, p. 39]{Bre73}. In this example, we estimate the interpolation function associated to the interpolation couple $\cN^{\partial\mathcal{E}} = (\norm[H]{\cdot} , |\partial\mathcal{E} (\cdot )|)$, and we identify the effective domain $\dom{(\cN^{\partial\mathcal{E}})_\BFS^\tau }$ at least for reasonable Banach function spaces $\BFS$ between $L^1 (0,\infty )$ and $L^2 (0,\infty )$. 

Given a Banach function space $\BFS$ over $(0,\infty )$, we define the space $\BFS^2$ by
\[
\BFS^2 := \{ g \in L^0 (0,\infty ) \st t\mapsto \frac{g(\sqrt{t})}{\sqrt{t}} \in \BFS \} ,
\]
which becomes a Banach space for the natural norm
\[
\Norm[\BFS^2]{g} := \Norm[\BFS]{t\mapsto \frac{g(\sqrt{t})}{\sqrt{t}}}
\]
For example, if $\BFS_{\theta ,p} = L^p (0,\infty ; t^{p(1-\theta)-1}\ud t)$ is a polynomially weighted $L^p$ space, then $\BFS_{\theta ,p}^2 = \BFS_{2\theta ,p}$ is again a polynomially weighted $L^p$ space. The latter space is a Banach space of measurable functions, and it is a Banach function space in the sense of our definition if and only if $\theta\in (0,\frac12 )$. 

In the following theorem, we assume that the dilation operator $D_2$ given by $(D_2 f) (t) := f(2t)$ is bounded on $\BFS$. The dilation operator $D_2$ is bounded on the polynomially weighted $L^p$ spaces $\BFS_{\theta ,p}$ and $\Norm[\mathcal{L} (\BFS_{\theta ,p} )]{D_2} = \frac{1}{2^{1-\theta}}$. Note that the norm of the dilation operator on $\BFS_{\theta ,p}$ is strictly less than $\frac{1}{\sqrt{2}}$ if $\theta\in (0,\frac12 )$.
 
 \begin{theorem}{\label{Th p1}}
 Let  $\mathcal{E}:H\to [0,\infty]$ be a proper, lower semicontinuous and convex function, and let $\partial\mathcal{E}$ be its subgradient. Consider the interpolation couples of functions $\cN^{\partial\mathcal{E}} = (\norm[H]{\cdot} , |\partial\mathcal{E} (\cdot )|)$ and $\cN^{\sqrt{\mathcal{E}}} = (\norm[H]{\cdot} , \sqrt{\mathcal{E}} )$. Let $\BFS$ be a Banach function space over $(0,\infty )$ be such that $\BFS^2$ is a Banach function space, too, and such that the dilation operator $D_2$ is bounded on $\BFS$ and that $\gamma := \sqrt{2} \Norm[\mathcal{L} (\BFS )]{D_2} < 1$. Then, for every $\tau >0$ and for every $x\in H$,  
 \begin{multline} {\label{eq p.1}}
 \frac{1}{\sqrt{2}}\, \Norm[\BFS]{t \mapsto \frac{\Norm[H]{x - J_tx}}{t} \, \chi_{(0,\frac{\tau^2}{2})} (t)} \leq (\cN^{\sqrt{\mathcal{E}}})_{\BFS^2}^\tau (x) \\
   \leq  \frac{1}{1-\gamma} \, \Norm[\BFS]{t \mapsto \frac{\Norm[H]{x - J_tx}}{t} \, \chi_{(0,2\tau^2)} (t) }  
 +  \frac{\gamma}{1-\gamma} \, \Norm[\BFS]{t\mapsto \sqrt{\frac{\mathcal{E}(J_tx)}{t}} \, \chi_{(\tau^2,2\tau^2 )} (t)} .
\end{multline} 
In particular,
\[
\dom{(\cN^{\partial\mathcal{E}})_\BFS^\tau } = \dom{(\cN^{\sqrt{\mathcal{E}}} )_{\BFS^2}^\tau} .
\]
 \end{theorem}
 
 \begin{proof}
 For every $x$, $v\in H$ and for every $t>0$, 
 \begin{align*}
  \mathcal{E}(v) -\mathcal{E}(J_tx)
 & \geq \langle \frac{x-J_tx}{t}, v-J_tx\rangle_H \\
 & \geq \frac{\Norm[H]{x-J_tx}^2}{t} - \frac{1}{t}\Norm[H]{x-J_tx} \Norm[H]{v-x}\\
 & \geq \frac{\Norm[H]{x-J_tx}^2}{2t} - \frac{\Norm[H]{v-x}^2}{2t} ,
 \end{align*}
 that is,
 \begin{equation}{\label{p.2}}
  \mathcal{E}(J_tx) + \frac{\Norm[H]{x-J_tx}^2}{2t} \leq \mathcal{E}(v) + \frac{\Norm[H]{v-x}^2}{2t} .
 \end{equation}
 Notice in passing that this is a proof that the infimum in the definition of the Moreau-Yosida approximation $\mathcal{E}_t (x) := \inf_{v\in H} \mathcal{E} (v) + \frac{\Norm[H]{v-x}^2}{2t}$ is attained, and that $J_t x$ is a global minimizer.
 
 The preceding inequality and the positivity of $\mathcal{E}$ imply that for every $x$, $v\in H$,
\begin{equation*}
\Norm[H]{x-J_tx} \leq \sqrt{2t\mathcal{E}(v)} + \Norm[H]{v-x} ,
\end{equation*} 
and hence
\begin{equation*}
 {\Norm[H]{x-J_tx}} \leq \inf_{v\in H} \left( \sqrt{2t\mathcal{E}(v)} + \Norm[H]{v-x} \right) = K^{\sqrt{\mathcal{E}}} (x,\sqrt{2t}) .
 \end{equation*}
We divide this inequality by $t$, multiply by $\chi_{(0,\frac{\tau^2}{2} )} (t)$, take the norm in $\BFS$, and use the definition of $\BFS^2$:
\begin{align*}
\lefteqn{\Norm[\BFS]{t \mapsto \frac{\Norm[H]{x - J_tx}}{t}\chi_{(0,\frac{\tau^2}{2})(t)}} } \\
& \leq 2\, \Norm[\BFS]{t \mapsto \frac{K^{\sqrt{\mathcal{E}}} (x,\sqrt{2t})}{2t}\chi_{(0,\frac{\tau^2}{2})(t)}} \\
& \leq 2\, \Norm[\mathcal{L} (\BFS )]{D_2} \Norm[\BFS]{t \mapsto \frac{K^{\sqrt{\mathcal{E}}} (x,\sqrt{t})}{t}\chi_{(0,\tau^2)(t)}} \\
&\leq \sqrt{2}\Norm[\BFS^2]{t \mapsto \frac{K^{\sqrt{\mathcal{E}}} (x,t)}{t}\chi_{(0,\tau)(t)}}\\
& \leq \sqrt{2} (\cN^{\sqrt{\mathcal{E}}})_{\BFS^2}^\tau (x) .
\end{align*}
This proves the first inequality of \eqref{eq p.1}.

Now for the second inequality of \eqref{eq p.1}, we have from \eqref{p.2} for all $x$, $v\in H$,
\begin{equation*}
\mathcal{E}(J_tx) \leq \mathcal{E}(v) + \frac{\Norm[H]{v-x}^2}{2t} ,
\end{equation*} 
and in particular, when we choose $v=J_{2t}x$ and multiply both sides by $\frac{1}{t}$, then
\begin{align*}
\frac{1}{t}\mathcal{E}(J_tx) \leq \frac{1}{t}\mathcal{E}(J_{2t}x) + \frac{\Norm[H]{J_{2t}x-x}^2}{2t^2} 
\end{align*}
We take square roots on both sides of the inequality, multiply the resulting inequality by the characteristic function $\chi_{(0,\tau^2 )}$ and then take the norm in $\BFS$:
\begin{align*}
\lefteqn{\Norm[\BFS]{t\mapsto \sqrt{\frac{\mathcal{E}(J_tx)}{t}} \, \chi_{(0,\tau^2 )} (t)}} \\
& \leq \sqrt{2} \, \Norm[\BFS]{t\mapsto \sqrt{\frac{\mathcal{E}(J_{2t} x)}{2t}} \, \chi_{(0,\tau^2 )} (t)} +  \sqrt{2}\, \Norm[\BFS]{t\mapsto \frac{\Norm[H]{J_{2t}x-x}}{2t} \, \chi_{(0,\tau^2 )} (t)} \\
& \leq \sqrt{2} \Norm[\mathcal{L} (\BFS )]{D_2} \, \Norm[\BFS]{t\mapsto \sqrt{\frac{\mathcal{E}(J_tx)}{t}} \, \chi_{(0,2\tau^2 )} (t)} + \sqrt{2} \Norm[\mathcal{L} (\BFS )]{D_2} \, \Norm[\BFS]{t\mapsto \frac{\Norm[H]{J_{t}x-x}}{t} \, \chi_{(0,2\tau^2 )} (t)} 
\end{align*}
By assumption, $\gamma = \sqrt{2} \Norm[\mathcal{L} (\BFS )]{D_2} <1$, and therefore,
\begin{align*}
\lefteqn{\Norm[\BFS]{t\mapsto \sqrt{\frac{\mathcal{E}(J_tx)}{t}} \, \chi_{(0,\tau^2 )} (t)}} \\
& \leq \frac{\gamma}{1-\gamma} \, \Norm[\BFS]{t\mapsto \sqrt{\frac{\mathcal{E}(J_tx)}{t}} \, \chi_{(\tau^2,2\tau^2 )} (t)} + \frac{\gamma}{1-\gamma }\, \Norm[\BFS]{t\mapsto \frac{\Norm[H]{J_{t}x-x}}{t} \, \chi_{(0,2\tau^2 )} (t)} .
\end{align*}
By definition of the interpolation function, by definition of the norm in $\BFS^2$, and by estimating the infimum in the definition of the $K$-functional, one has
\begin{align*}
   (\cN^{\sqrt{\mathcal{E}}})_{\BFS^2}^\tau (x) & = \Norm[\BFS]{t\mapsto \frac{K^{\sqrt{\mathcal{E}}} (x,\sqrt{t})}{t} \, \chi_{(0,\tau^2 )} (t)} \\
   & \leq \Norm[\BFS]{t\mapsto \left( \frac{\Norm[H]{J_{t}x-x}}{t} + \sqrt{\frac{\mathcal{E}(J_tx)}{t}} \right) \, \chi_{(0,\tau^2 )} (t)} \\
   & \leq \frac{1}{1-\gamma} \, \Norm[\BFS]{t\mapsto \frac{\Norm[H]{J_{t}x-x}}{t} \, \chi_{(0,2\tau^2 )} (t)} \\
   & \phantom{\leq\ } + \frac{\gamma}{1-\gamma} \, \Norm[\BFS]{t\mapsto \sqrt{\frac{\mathcal{E}(J_tx)}{t}} \, \chi_{(\tau^2,2\tau^2 )} (t)} .
\end{align*}
This is the second inequality in \eqref{eq p.1}. Note that the second term on the right hand side of the previous inequality is finite for every $x\in H$; in fact, the function $t\mapsto \sqrt{\frac{\mathcal{E}(J_tx)}{t}}$ is bounded on the interval $(\tau^2 , 2\tau^2 )$, and $L^1\cap L^\infty (0,\infty )\subseteq\BFS$. This observation and Corollary \ref{th4.3} imply the equality of the effective domains. 
\end{proof}

\paragraph{\bf Application.} Let $\Omega \subseteq\mathbb{R}^n$ be a bounded, smooth domain, let $q\in (1,\infty )$, and let $\mathcal{E}:L^2(\Omega)\to [0,\infty]$ be the energy of the negative Dirichlet-$q$-Laplace operator given by
\[ 
\mathcal{E}(u) = \begin{cases}
         \frac{1}{q}\int_{\Omega} |\nabla u|^q & \text{if } u\in \mathring{V}^{1}_{q,2} (\Omega ) , \\
        \infty & \text{otherwise} ,
        \end{cases}  
\] 
where $\mathring{V}^{1}_{q,2} (\Omega )$ is the closure of the space of test functions $C^\infty_c (\Omega )$ in the Sobolev space
\[
V^{1}_{q,2} (\Omega ) := \{ u\in L^2 (\Omega ) \st \nabla u \in L^q (\Omega )^n \} ,
\]
and where $\nabla u$ is the distributional gradient of $u\in L^2 (\Omega )$. The space $V^{1}_{q,2} (\Omega )$ is a Banach space for the norm $\norm[V^1_{q,2} (\Omega )]{u} = \norm[L^2 (\Omega )]{u} + \norm[L^q (\Omega )^n]{\nabla u}$. If $q\geq 2$, then $V^1_{q,2} (\Omega ) = W^{1,q} (\Omega )$ with equivalence of norms, as one may for example show by applying the Poincar\'e-Wirtinger inequality and by using the boundedness of $\Omega$. By the first Poincar\'e inequality, and for $q\in [2,\infty )$, the application $u\mapsto \norm[L^q(\Omega )^n]{\nabla u}$ is an equivalent norm on $\mathring{V}^{1}_{q,2} (\Omega ) = \mathring{W}^{1,q} (\Omega )$.

The energy of the negative Dirichlet-$q$-Laplace operator is densely defined, lower semicontinuous and convex. Its subgradient $-\Delta_q := \partial\mathcal{E}$ is a realization in $L^2 (\Omega )$ of the negative $q$-Laplace operator with homogeneous Dirichlet boundary condition $u = 0$. The negative Dirichlet-$q$-Laplace operator is $m$-accretive and densely defined on $L^2 (\Omega )$. 

\begin{corollary} \label{cor.qlaplace}
Let $-\Delta_q$ be the negative Dirichlet-$q$-Laplace operator on $L^2 (\Omega )$ ($q\in [2,\infty )$), and let $\cN^{-\Delta_q} = (\norm[L^2]{\cdot} , |\Delta_q \cdot |)$ be the associated interpolation couple. Then, for every $\theta\in (0,\frac12 )$ and every $p\in (1,\infty )$,
\[
\dom{(\cN^{-\Delta_q})_{\BFS_{\theta ,p}}^\tau} = (L^2 (\Omega ) , \mathring{W}^{1,q} (\Omega ))_{\alpha ,r} ,
\]
where $\BFS_{\theta ,p} = L^p (0,\infty ; t^{p(1-\theta ) -1} \ud t)$ is the polynomially weighted $L^p$ space, $\alpha = \frac{q\theta}{1+\theta (q-2)}$ and $r=p(1+\theta(q-2))$.
\end{corollary}

\begin{proof}
As mentioned in the paragraph before Theorem \ref{Th p1}, $\BFS_{\theta ,p}^2 = \BFS_{2\theta ,p}$, and hence, by applying successively Theorem \ref{Th p1} and Lemma \ref{3.11}, 
\[
\dom{(\cN^{\partial\mathcal{E}})_{\BFS_{\theta,p}}^\tau } = \dom{(\cN^{\sqrt{\mathcal{E}}} )_{\BFS_{2\theta ,p}}^\tau} = \dom{(\cN^{\sqrt{\mathcal{E}}} )_{\BFS_{2\theta ,p}}} .
\]
By definition, the space on the right hand side is the space of all $u\in L^2 (\Omega )$ such that $t\mapsto \frac{K(u,t)}{t}$ is an element of $\BFS_{2\theta ,p}$, where
\[
K(u,t) = \inf_{v\in \mathring{W}^{1,q} (\Omega )} ( \norm[L^2 (\Omega )]{u} + \frac{t}{\sqrt{q}} \norm[L^q (\Omega )^n]{\nabla u}^{\frac{q}{2}} .
\]
In other words, the space on the right hand side coincides with the power space $(L^2 (\Omega ) , (\mathring{W}^{1,q} (\Omega ))^{\frac{q}{2}} )_{2\theta ,p}$ defined in \cite[Section 3.11, p. 68]{BeLo76}. By the power theorem (see \cite[Theorem 3.11.6]{BeLo76}), 
\begin{align*}
\dom{(\cN^{-\Delta_q})_{\BFS_{\theta ,p}}^\tau} & = (L^2 (\Omega ) , \mathring{W}^{1,q} (\Omega ))_{\alpha ,r} ,
\end{align*}
where $\alpha = \frac{q\theta}{1+\theta (q-2)}$ and $r=p(1+\theta(q-2))$. 
\end{proof}

Consider the initial boundary value problem
\begin{equation} \label{eq.qlaplace}
\begin{split}
 \partial_t u - \Delta_q u & = 0 \text{ in } (0,T)\times\Omega , \\
 u & = 0 \text{ in } (0,T)\times\partial\Omega , \\
 u(0,\cdot ) & = u_0 \text{ in } \Omega .
\end{split}
\end{equation}
This problem can abstractly be rewritten as a gradient system on $L^2 (\Omega )$ by using the energy of the Dirichlet-$q$-Laplace operator and its subgradient. By the Crandall-Liggett theorem, for every initial value $u_0\in L^2 (\Omega )$ the above problem admits a unique mild solution $C ([0,T];L^2 (\Omega ))$. Actually, the semigroup generated by $\Delta_q$ is regularizing by \cite[Th\'eor\`eme 3.2]{Bre73}, and every mild solution is a strong solution satisfying $u\in H^1_{loc} ((0,T];L^2 (\Omega ))$. Moreover, if $u_0\in \mathring{V}^{1}_{q,2} (\Omega )$ (the effective domain of the energy), then 
\[
\| \partial_t u\|_{L^2 (\Omega )} , \, \| \Delta_q u\|_{L^2 (\Omega )} \in L^2 (0,T) ,
\]
which corresponds to a kind of $L^2$ maximal regularity. With the help of Theorem \ref{th5.3} and Corollary \ref{cor.qlaplace}, we obtain the following variant.

\begin{corollary}
Let $\theta\in (0,\frac12 )$ and $p\in (1,\infty )$, and set $\alpha = \frac{q\theta}{1+\theta (q-2)}$ and $r=p(1+\theta(q-2))$. Let $u_0\in L^2 (\Omega )$ and let $u\in C([0,T] ; L^2 (\Omega ))$ be the unique mild solution of the initial boundary value problem \eqref{eq.qlaplace} with $q\in [2,\infty )$. Then the solution $u$ of \eqref{eq.qlaplace} satisfies
\[
\| \partial_t u\|_{L^2 (\Omega )} , \, \| \Delta_q u\|_{L^2 (\Omega )} \in L^p (0,T ; t^{p(1-\theta )-1} \; \ud t) 
\]
if and only if
\[
u_0 \in (L^2 (\Omega ) , \mathring{W}^{1,q} (\Omega ))_{\alpha ,r} .
\]
\end{corollary}

\nocite{Tar72a,Tar72}
\nocite{Br69}
\nocite{LoSh68,LoSh69}
\nocite{CoHa18}
\nocite{Pe70b,Pe70}

\bibliographystyle{plain}
\def\cprime{$'$}
  \def\ocirc#1{\ifmmode\setbox0=\hbox{$#1$}\dimen0=\ht0 \advance\dimen0
  by1pt\rlap{\hbox to\wd0{\hss\raise\dimen0
  \hbox{\hskip.2em$\scriptscriptstyle\circ$}\hss}}#1\else {\accent"17 #1}\fi}
  \def\cprime{$'$} \def\cprime{$'$} \def\cprime{$'$}

\end{document}